\documentclass[12pt]{article} 

\usepackage{latexsym}
\usepackage{indentfirst}
\usepackage{graphicx}
\usepackage{amsmath}
\usepackage{amssymb}
\usepackage{amsthm}
\usepackage{bm}
\usepackage{accents}
\usepackage{enumitem}
\usepackage{relsize}
\usepackage{datetime}
\usepackage[toc,page]{appendix}
\usepackage{graphicx}
\usepackage{float}
\usepackage{upgreek}
\usepackage{bigints}
\usepackage{bbm}
\usepackage{sidecap}
\usepackage{caption}
\usepackage{xfrac}
\usepackage{kantlipsum}
\allowdisplaybreaks

\usepackage[numbered]{bookmark}

\let\oldproofname=\proofname
\renewcommand{\proofname}{\textsc{\oldproofname}}

\makeatletter
\newcommand{\mylabel}[2]{#2\def\@currentlabel{#2}\label{#1}}
\makeatother

\newcommand{\vertiii}[1]{{\left\vert\kern-0.25ex\left\vert\kern-0.25ex\left\vert #1 
    \right\vert\kern-0.25ex\right\vert\kern-0.25ex\right\vert}}

\usepackage{hyperref}
\usepackage{url}
\hypersetup{pdftex,colorlinks=true,linkcolor=blue,urlcolor=blue,citecolor=blue}
\usepackage{hypcap}

\makeatother
 \makeatletter
\newcommand{\mathleft}{\@fleqntrue\@mathmargin50pt}
\newcommand{\mathcenter}{\@fleqnfalse}
\makeatother

\makeatother
 \makeatletter
\newcommand{\zeromathleft}{\@fleqntrue\@mathmargin0pt}
\makeatother

\newtheorem{prop}{Proposition}[section]
\newtheorem*{defi*}{Definition}
\newtheorem{defi}[prop]{Definition}

\newtheorem*{lem*}{Lemma}
\newtheorem{rem}[prop]{Remark}
\newtheorem{thm}[prop]{Theorem}

\newtheorem{claim}{Claim}

\numberwithin{equation}{section}

\usepackage{color}
\definecolor{mygray}{gray}{0.4}

\DeclareMathOperator{\sign}{sign}

\begin{document}


\noindent {\Large\bf  The nonlinear Schr\"{o}dinger equation with white noise dispersion on quantum graphs}\\

\noindent Iulian C\^{i}mpean \footnote{\label{imar}\textsc{"Simion Stoilow" Institute of Mathematics  of the Romanian Academy, Calea Grivi\c{t}ei Street, No. 21, 010702, Bucharest, Romania}\\
\textsl{E-mail address:} \texttt{iulian.cimpean@imar.ro}} and
Andreea Grecu \footnote{\textsc{University of Bucharest, Academiei Street, No. 14, 010014, Bucharest, Romania} \\
\textsl{E-mail address}: \texttt{andreea.grecu@my.fmi.unibuc.ro}}\textsuperscript{,\ref{imar}}
\\

\noindent
{\small {\bf Abstract.} We show that the nonlinear Schr\"{o}dinger equation (NLSE) with white noise dispersion on quantum graphs is globally well-posed in $L^2$ once the free deterministic Schr\"{o}dinger group satisfies a natural $L^1-L^{\infty}$ decay, which is verified in many examples. Also, we investigate the well-posedness in the energy domain in general and in concrete situations, as well as the fact that the solution with white noise dispersion is the scaling limit of the solution to the NLSE with random dispersion.}
 \\[2mm]




\section{Introduction} \label{introduction}

The aim of this paper is to study well-posedeness results for nonlinear Schr\"odinger equation (NLSE) on quantum graphs, of the type
\begin{equation} \label{eq general}
\mathrm{i} \dfrac{\partial u}{\partial t} +m(t) \dfrac{\partial^{2}u}{\partial x^{2}} + f(|u|^2)u=0,
\end{equation}
where the dispersion coefficient $m$ is randomly varying.
In order to give a precise description of the goals of this work we need some preliminaries on quantum graphs, so let us postpone these to the beginning of Section 2, and discuss here some more or less general aspects concerning \eqref{eq general} that are relevant to the present work. 
First of all, recall that on $\mathbb{R}^d$, equation \eqref{eq general} is generically used to model the evolution of nonlinear dispersive waves in inhomogeneous media.
It is important to mention that depending on the physical interpretation, the usual time variable $t$ and space variable $x$ can have reversed meaning. 
For example, when equation \eqref{eq general} describes the wave function of a particle in a Bose-Einstein condensate (also known as Gross-Pitaevskii equation), or stands as a model for superconductivity (known as Ginzburg-Landau equation), the variables $t$ and $x$ denote time and space, respectively; see \cite{berge}, but also the introduction of \cite{AdNo09} and the references therein.
In contrast, when modeling the propagation of a signal in an optical fiber ($d=1$), $t$ denotes the distance along the fiber, whilst $x$ corresponds to the retarded time; see e.g. \cite{agrawalarticle}. 

In fact, the present work is inspired by the one developed in \cite{Ma06} and \cite{debouarddebussche} which concerns the latter physical model, so let us give more details in this direction: The following nonlinear Schr\"odinger equation on $\mathbb{R}$
\begin{equation} \label{eq 1.1}
\mathrm{i} \dfrac{\partial v}{\partial t} + \varepsilon m(t) \dfrac{\partial^{2}v}{\partial x^{2}} + \varepsilon^{2}|v|^2v=0, \quad x \in \mathbb{R}, \ t>0
\end{equation}
was considered in \cite{Ma06} and \cite{debouarddebussche} as a model for the propagation of a signal in an optical fiber with randomly varying dispersion, where $(m(t))_{t\geq 0}$ is a centered stationary stochastic process which models the fluctuations of the dispersion, while $\varepsilon \ll 1$ controls its amplitude. 

In order to understand the diffusion approximation for (\ref{eq 1.1}), i.e. the limiting case $\varepsilon\to 0$, it is convenient to consider the scaling $X(t,x)=v(\frac{t}{\varepsilon^{2}},x)$, which leads to
\begin{equation} \label{eq 1.2}
\mathrm{i} \dfrac{\partial X}{\partial t} + \dfrac{1}{\varepsilon} m(\dfrac{t}{\varepsilon^{2}}) \dfrac{\partial^{2}X}{\partial x^{2}} + |X|^2X=0, \quad x \in \mathbb{R}, \ t>0
\end{equation}
If the following invariance principle is in force
\begin{enumerate}
\item[\mylabel{H0}{\bf (H.0)}] $
\bigg(\mathop{\mathlarger{\int_0^{t}}} \dfrac{1}{\varepsilon} m \bigg( \dfrac{s}{\varepsilon^2} \bigg) \ ds\bigg)_{t\geq 0}$ converges in law to a Brownian motion $(\beta(t))_{t\geq 0}$ when $\varepsilon \to 0$,
\end{enumerate}

\noindent then the limit equation reads
\begin{equation} \label{eq 1.3}
 dX(t)=\mathrm{i} \Delta X(t)\circ d\beta(t) + \mathrm{i} |X|^{2}X(t) \; dt, \quad t\geq 0
\end{equation}
or in It\^{o} form
\begin{equation} \label{eq 1.4}
dX(t)=-\frac{1}{2}\Delta^2 X(t) \; dt + \mathrm{i} \Delta X(t) \; d \beta(t) + \mathrm{i} |X|^{2}X(t) \; dt, \quad t\geq 0
\end{equation}

In fact, this model was first studied in \cite{Ma06} with a truncated nonlinearity $f(|v|^2)v$ instead of the cubic one, where $f$ is a smooth cutoff of the identity. 
More precisely, it was shown that in the case of such a truncation, a contraction argument works out smoothly to prove that equations (\ref{eq 1.2}) are well posed in $L^{2}(\mathbb{R})$ and $H^{2}(\mathbb{R})$, and their solutions converge in distribution to the $L^{2}$-solution of the (truncated) limit equation (\ref{eq 1.4}), when $\varepsilon \to 0$.
Also, a numerical splitting scheme was developed to simulate these solutions.

The case of full nonlinearity is much more involved, and it was treated a few years later in \cite{debouarddebussche} and then in \cite{DeTs11}, for quintic nonlinearity.
First of all, the varying dispersion makes the Hamiltonians associated to equations (\ref{eq 1.1}) or (\ref{eq 1.4}) to be no longer preserved, hence there are no a priori energy estimates for the solutions. 
Fortunately, the $L^{2}$-norm of the solution is preserved, and an $L^{2}$-approach turned out to be suitable, yet much more delicate than in the truncated case.
More precisely, the key ingredient from \cite{debouarddebussche}, \cite{DeTs11} consists of Strichartz type estimates obtained for the solution to the linear version of equation (\ref{eq 1.4}) (i.e. when the nonlinear term is discarded). 
These estimates are then employed to solve the nonlinear equation (\ref{eq 1.4}) in mild form on $L^{2}(\mathbb{R})$ (even on $H^{1}$), through a fixed point argument.
We emphasize that the Strichartz estimates were obtained in \cite{debouarddebussche}, \cite{DeTs11} for white noise dispersion only.
They are not available for the (linear) equations (\ref{eq 1.2}) with random dispersion in general, and for this reason, for full nonlinearity, equations (\ref{eq 1.2}) can be solved only locally, up to some stopping time $\tau_\varepsilon$.
However, it was shown in \cite{debouarddebussche}, \cite{DeTs11} that these local solutions converge in law to the global solution of equation (\ref{eq 1.4}), when $\varepsilon \to 0$.

Recently, in \cite{ChGu15} and \cite{gubinellichouk}, global well-posedness of nonlinear PDEs with modulated dispersion have been extended to a larger class of "sufficiently irregular" noises.
Although the results obtained in the present paper could be reconsidered under such more general noises, for simplicity we shall treat only the case of white noise dispersion.

Concerning the motivation of this paper, let us mention that in recent years there has been a growing interest in studying nonlinear Schrodinger equations on ramified structures, with different applications: condensed matter physics, nonlinear fiber optics, hydrodynamics, fluid transport, or neural networks.
Such ramified structures are modeled by quantum graphs, i.e. metric graphs endowed with a self-adjoint differential operator; see e.g. \cite{berkolaiko} and the references therein.
For example, in the standard case of propagation of waves in one dimensional domains, the presence of spatial point defects lead to boundary conditions and hence to a quantum graph model; see \cite{holmermarzuolasworski, AdNo09, kovariksacchetti2010, banicaignat2014} for details and further applications. 
Returning to the original applications of equations \eqref{eq 1.2}-\eqref{eq 1.4} to optical fibers as considered in \cite{Ma06} and \cite{debouarddebussche}, recall that the time and space variables $t$ and $x$ have reversed meanings.
Nevertheless, in the typical situation of bimodal optical fibers where two solitons (one narrower and one wider) are coupled by two corresponding NLS equations, it is convenient to reduce the system to a single NLS equation in which the narrow soliton in the mate mode will be represented by a perturbation of a one dimensional NLS equation with an effective delta potential in the (temporal) $x$ variable; see \cite[eq.(1)]{caomalomed} and the refences therein. 
In this way, we are once again lead to boundary conditions at $x=0$, hence to a quantum graph model where $t$ is the distance along the fiber while $x$ corresponds to the delayed time. 
More complex optical networks have been recently considered in \cite{GnSmDe11} or \cite{AdCaFiNo12}.

The goal of this paper is to investigate the well-posedeness of equation (\ref{eq 1.3}) (or \eqref{eq 1.4} in equivalent Ito form), as well as the convergence of the solutions of (\ref{eq 1.2}) when $\varepsilon \to 0$, on general quantum graphs, denoted further by $\Gamma$.
In the spirit of the previously mentioned applications, we emphasize that the modulation of the dispersion coefficient $(m(t))_{t\geq 0}$ in \eqref{eq 1.3} or $(\beta(t))_{t\geq 0}$ in \eqref{eq 1.4} may have either temporal or spatial meaning, depending on the physical system under consideration.

The rest of this paper is structured in two sections: in Section 2 we start with an overview on quantum graphs, and state the precise goals of this work. The rest and most consistent part of this section is a systematic exposition of the main results of the paper, trying to point out the specific difficulties which we deal with, in contrast to the work in \cite{Ma06}, \cite{debouarddebussche}, or \cite{DeTs11}; we structure this exposition in four subsections, as follows: Subsection 2.1 is devoted to the well-posedness of the linear equation \eqref{linear stochastic} (Proposition \ref{linear solution}) and the corresponding Strichartz estimates (Theorem \ref{stochastic strichartz}); in Subsection 2.2 we obtain the well-posedness of equation \eqref{eq 2.6} on $L^{2}(\Gamma)$ (Theorem \ref{wellposednessL2}); in Subsection 2.3, Theorems \ref{well-posedness} and \ref{convergence}, we state the well posedness in the energy domain and convergence of the approximate solutions when $\varepsilon \to 0$, for truncated nonlinearities; the extension of these last two results to full nonlinearities is much more delicate, and we restrict our study to the case of star-graphs, with the mention that our strategy is general and can be applied to other situations (see Subsection 2.4, Theorems \ref{thm 2.16}, \ref{convergence3}, and \ref{thm 2.19}). 
Section 3 contains the proofs of the results presented in Section 2.

\section{Preliminaries on quantum graphs and the main results} \label{preliminaries and main results}

First, let us give a brief overview on quantum graphs, following mainly \cite{laplacians} and \cite{berkolaiko}. 
A finite graph $\Gamma$ is a triplet $\Gamma=(V,E,\partial)$, where $V=\{v_i \}_{i}$ is a finite set of vertices, $E=\{e_j \}_j =:I \cup \mathcal{E}$ is a finite set of (internal, respectively external) edges that connect the vertices, and $\partial$ is a map (called orientation) which assigns to an internal edge $e \in I$ an ordered pair of vertices $\partial(e)=\{\partial^{-}(e), \partial^{+}(e)\}$, and to an external edge $e \in \mathcal{E}$ its single vertex. $\partial^{-}(e)$ and $\partial^{+}(e)$ are called the initial and the terminal vertex of the edge $e$, respectively. The graph $\Gamma$ is assumed to be connected, i.e. any two vertices can be connected by an edge w.r.t. the order given by $\partial$.

We endow the graph $\Gamma$ with the following metric structure: each internal edge $e$ is identified with an interval $[0,l_e]$ where zero corresponds to $\partial^{-}(e)$; similarly, an external edge corresponds to a semi line $[0,\infty)$. Based on this identification, each edge is endowed with the euclidean metric on the corresponding interval, and, in general, the distance between two points on $\Gamma$ is taken to be the length of the shortest path between them.

\vspace{0.2 cm}
\noindent {\bf Function spaces on $\Gamma$}. A complex valued function $f : \Gamma \to \mathbb{C}$ is regarded as a collection $f=(f_e)_{e \in E}$, where $f_e : I_e \to \mathbb{C}$; $I_e=[0,l_e]$ or $[0,\infty)$, for all $e \in E$.
We denote by $L^p(\Gamma)$, $1 \leq p \leq \infty$, the space of all elements $f=(f_e)_{e \in E}$ where $f_e \in L^p(I_e)$ for all $e \in E$. $L^p(\Gamma)$ becomes a Banach space with respect to the norm
\begin{equation*}
|| f ||^p_{L^p(\Gamma)} = \sum\limits_{e \in E} || f_e ||^p_{L^p(I_e)} \quad {\rm if} \quad p < \infty,\quad  || f ||_{L^{\infty}(\Gamma)} = \sup\limits_{e \in E} || f_e ||_{L^{\infty}(I_e)},
\end{equation*}

\noindent i.e. $L^p(\Gamma)=\bigoplus\limits_{e \in E} L^p(I_e)$. Similarly, we consider the Sobolev spaces $W^{k,p}(\Gamma)= \bigoplus W^{k,p}(I_e)$, $k\in \mathbb{N}^{\ast}, 1\leq p \leq \infty$ with the norm
\begin{equation*}
|| f ||^p_{W^{k,p}(\Gamma)} = \sum\limits_{e \in E} || f_e ||^p_{W^{k,p}(I_e)} .
\end{equation*}

\begin{rem}
Since $\Gamma$ is one dimensional, every element $f \in W^{k,p}(\Gamma)$ possesses a continuous version on each interval $I_e$, but there is no a priori information on how the values of $f$ at the vertices are coupled. The coupling conditions are provided by (the domain of) the heat operator, which we describe in the sequel.
\end{rem}

Let us first consider the space of test functions $D_0=\bigoplus\limits_{e \in E} C_0^{\infty}(I_e)$, where $C_0^{\infty}$ consists of infinitely differentiable functions with compact support on $\mathring{I_e}$, and the operator $\Delta_\Gamma^0 : D_0 \subset L^2(\Gamma) \to L^2(\Gamma)$, given by 
$$
(\Delta_\Gamma^0f)_e (x)=\dfrac{d^2}{dx^2} f_e (x), \;\; e \in E, x \in I_e.
$$
By symmetry, $(\Delta_\Gamma^0,D_0)$ is a closable operator on $L^2(\Gamma)$, and its minimal domain, i.e. the domain of its closure, is given by $D(\Delta_\Gamma^0):=\{ f \in W^{2,2}(\Gamma) : f_e (0)=f_e (l_e)= f'_e (0)=f'_e (l_e)=0, \ e \in E\}$. 
However, the interest is to consider other coupling conditions, especially those that correspond to self-adjoint extensions of $(\Delta_\Gamma^0,D_0)$, which clearly is not the case of $(\Delta_\Gamma^0,D(\Delta_\Gamma^0))$.
Fortunately, the self-adjoint extensions of $(\Delta_\Gamma^0,D_0)$ can be completely characterized in terms of the coupling conditions, as follows: let $\{A_v,B_v \}_{v \in V}$ be a family of matrices from $\mathbb{C}^{n_v \times n_v}$, where $n_v$ denotes the number of edges with common vertex $v \in V$. 
Let us consider the extension $\Delta_\Gamma$ of $(\Delta_\Gamma^0,D_0)$, with domain
\begin{equation*}
D(\Delta_\Gamma):=\{ f \in H^2(\Gamma): \ A_v \overline{f_v} + B_v \overline{f'_v}=0, \ v \in V \}
\end{equation*}
where $\overline{f_v}=(f_e(v))_{e \in E}$ and $\overline{f'_v}=(f'_e(v))_{e \in E}$ are column vectors.

\vspace{0.2 cm}
Now we can recall the following well known characterization.
\begin{thm}[cf. \cite{laplacians}] \label{self-adjointconditions} $(\Delta_\Gamma,D(\Delta_\Gamma))$ is self-adjoint if and only if the following conditions are satisfied:
\begin{enumerate}[label=(\roman*)]
\item The concatenated matrix $(A_v,B_v)$ has maximal rank, $n_v$;
\item $A_vB^{\dag}_v$ is self-adjoint, where $B^{\dag_v}$ is the adjoint transpose of $B_v$.
\end{enumerate}
\end{thm}

\noindent We emphasize that the coupling matrices $A_v,B_v$ are not unique (but merely modulo an invertible matrix).
There is another characterization due to \cite[Theorem 1.4.4]{berkolaiko}, according to which the operator $\Delta_\Gamma$ is self-adjoint if and only if for every vertex $v$ of degree $d_v$, there are three unique orthogonal (and mutually orthogonal) projectors $P_{D,v}, P_{N,v}$, and $P_{R,v}=I-P_{D,v}-P_{N,v}$ acting on $\mathbb{C}^{d_v}$, and $\Lambda_v$  invertible and self-adjoint acting on $P_{R,v}\mathbb{C}^{d_v}$, such that the boundary values of $f$ satisfy
\begin{equation*}
\begin{cases}
P_{D,v}\overline{f_v}=0, & "Dirichlet \, Part"\\
P_{N,v}\overline{f^{'}_v}=0, & "Neumann \, Part"\\
P_{R,v}\overline{f^{'}_v}=\Lambda_vP_{R,v}\overline{f_v}, & "Robin \, Part"\\
\end{cases}.
\end{equation*}

\noindent The quadratic form $\mathcal{E}$ associated to $\Delta_\Gamma$ is given by
\begin{align} \label{formdomain}
&\mathcal{E}(f,f) =  \sum_{j=1}^{|E|} \int_{e_j}  |f'_j(x)|^2 \, dx + \sum_{v \in V} < \Lambda_v P_{R,v} f(v), P_{R,v} f(v) > \; \; \mbox{for all } f\in D(\mathcal{E});  \nonumber \\
& D(\mathcal{E}):= \{  f \in W^{1,2}(\Gamma) \, : \, P_{D,v}f(v)=0 \; \mbox{ for all } v \in V  \}.
\end{align}

We will frequently employ the following equivalence of norms.
\begin{prop}[cf. \cite{berkolaiko}, p. 23]\label{equivnorms} There exists $M>0$ such that for any $f \in D(\mathcal{E})$, the norm $ \| \cdot \|_{D(\mathcal{E})}$ given by
\vspace{-10pt}
\begin{equation}\label{formnorm}
\| f \|_{D(\mathcal{E})} := \sqrt{M \| f \|^2_{L^2(\Gamma)} + \mathcal{E}(f,f)}
\end{equation}
is well defined and equivalent to $\|\cdot\|_{W^{1,2}(\Gamma)}$. 
In particular, $(D(\mathcal{E}), \| \cdot \|_{D(\mathcal{E})})$ is a Hilbert space.

\end{prop}

From now on we assume that $(\Delta_{\Gamma}, D(\Delta_{\Gamma})$ is a self-adjoint extension of $(\Delta^0_{\Gamma},D(\Delta^0_{\Gamma})))$ with local coupling conditions $\{ A_v,B_v \}_{v \in V}$. By $\mathrm{e}^{\mathrm{i}t \Delta_{\Gamma}}$ we denote the strongly continuous group of izometries on $L^2(\Gamma)$ generated by $\mathrm{i \Delta_{\Gamma}}$.

\vspace{0.2 cm}
\noindent
{\bf Main goals.} For the rest of the paper, $\beta:=(\beta(t))_{t\geq 0}$ denotes a standard 1-dimensional Brownian motion on a filtered probability space $(\Omega, \mathcal{F}, \mathcal{F}_t, \mathbb{P})$ which satisfies the usual hypotheses.
We point out that throughout, $(\beta(t))_{t\geq 0}$ could be replaced by $(\beta(t)+\mu t)_{t\geq 0}$ for some constant $\mu \in \mathbb{R}$.

The first main aim is to investigate the well-posedeness in $L^{2}(\Gamma)$ and $D(\mathcal{E})$ of the following  nonlinear Schr\"{o}dinger equation with white noise dispersion, on $\Gamma$:
\begin{equation} \label{eq 2.6}
\left\{
\begin{array}{lll}
                 dX(t)=-\frac{1}{2}\Delta_\Gamma^2 X(t) \; dt + i\Delta_\Gamma X(t) \; d \beta(t) + i|X|^{2\sigma}X(t) \; dt, & t\in[0,\infty)\\ [5pt]
                  X(0) =X_0.\\
         \end{array}
         \right.
\end{equation}
In fact, as in \cite{debouarddebussche}, the strategy is to tackle \eqref{eq 2.6} in mild form:
\begin{equation} \label{mild}
X(t)=S_{\beta}(t,0)X_0+\mathrm{i}\int_0^t S_{\beta}(t,s)|X(s)|^{2\sigma}X(s) \; ds, \quad t\geq 0, \ \mathbb{P}\mbox{-a.s.}
\end{equation}
where $S_\beta(\cdot,s)$ (see Subsection 2.1) gives the solution to the linear equation, starting at time $s\geq 0$.
We emphasize that the sign in front of the nonlinearity is not important, since $-\beta$ is still a Brownian motion.

The second aim is to study the convergence of the (local) solutions of 
\begin{equation} \label{approximate}
\begin{cases}
\dfrac{dX(t)}{dt} =\mathrm{i}  \dfrac{1}{\varepsilon} m\bigg(\dfrac{t}{\varepsilon^2}\bigg) \Delta_{\Gamma} X(t)+ \mathrm{i} |X|^{2\sigma}X(t) =0, & t>0, \ \mbox{ on } \Gamma\\
X(0)=X_0
\end{cases},
\end{equation}
to the solution of \eqref{eq 2.6}, when $\varepsilon \to 0$, provided that the process $m$ satisfies \ref{H0}.

\subsection{The linear stochastic equation and Strichartz type estimates} \label{linear equation and strichartz}

As in \cite{debouarddebussche}, the key ingredient to prove the well-posedness of \eqref{mild} is the Strichartz type estimates for the solution to the linear Schr\"{o}dinger equation with white noise dispersion:
\begin{equation}\label{linear stochastic}
\left\{
    \begin{array}{lll}
                  dX(t)=-\frac{1}{2}\Delta_\Gamma^2 X(t) \; dt + \mathrm{i}\Delta_\Gamma X \; d \beta(t), & t\geq s \\[5pt]
                  X(s) =X_s \\
	\end{array}
	\right.
\end{equation} 
Recall that on $\mathbb{R}$, the solution to \eqref{linear stochastic} can be explicitly obtained by Fourier transform (see \cite{Ma06} and \cite{debouarddebussche}), and it is given by
\begin{equation}\label{stochastic semigroup}
S_{\beta}(t,s)X_s(\omega):={\mathrm{e}^{\mathrm{i}[\beta(t)-\beta(s)](\omega)\Delta_\Gamma}}X_s(\omega) \;\; \mbox{ for }  s\leq t, \omega \in \Omega.
\end{equation}
Although Fourier transform is no longer available on $\Gamma$, we can use spectral arguments to rigorously show that $S_{\beta}(t,s)$ is well-defined and gives the solution to \eqref{linear stochastic}. 
More precisely, we have:
  
\begin{prop}\label{linear solution}
Let $s \geq 0$ and assume that $X_s\in D(\Delta_\Gamma^2)$ $\mathbb{P}$-a.s. Then $(S_{\beta}(t,s)X_s)_{t\geq s}$ given by \eqref{stochastic semigroup} is the (pathwise) unique strong solution to \eqref{linear stochastic}, with paths in $C([s,\infty),D(\Delta_\Gamma^2))$ a.s.
In particular, $|S_{\beta}(t,s)X_s|_{L^2}=|X_s|_{L^2} \ for \ all \ t \geq s \mbox{ a.s.}$
\end{prop}

\begin{rem}
In the previous proposition, if $X_s\in D(\Delta_\Gamma^2)$ is deterministic, then the fact that $S(\cdot,s)X_s$ is a solution would follow directly by It\^ o formula on Hilbert spaces (see e.g. \cite{daprato, liurockner}), since
$F:\mathbb{R}\rightarrow L^2(\Gamma)$, $F(t)={\mathrm{e}^{\mathrm{i}t\Delta_\Gamma}X_s}$ is twice Frechet differentiable. 
However, for the sake of the mild formulation \eqref{mild}, it is necessary to allow random initial data $X_s$, and we do this rigorously in Proposition \ref{linear solution} is to do this rigorously. ensures that $S(\cdot,s)X_s$ remains the solution to \eqref{linear stochastic} also for random initial data.
\end{rem}

In orther to get the desired Strichartz estimates, we point out that in the case $\Gamma = \mathbb{R}$, the key starting point in \cite{debouarddebussche} is the dispersive estimate $|e^{it\Delta}|_{L^{1}(\mathbb{R})\to L^{\infty}(\mathbb{R})} \lesssim |t|^{-1/2} ,t\in \mathbb{R}^{\ast}$.
Such an estimate on $\Gamma$ is verified in few situations, mainly because of the presence of nonempty point spectrum, and
it turns out that the following general hypothesis is much more convenient:
\begin{enumerate}
\item[\mylabel{H1}{\bf (H.1)}] The number of eigenvalues of $-\Delta_\Gamma$, counting their multiplicities, is at most finite, and there exists a constant $C > 0$ such that
\begin{equation} \label{H1eq}
\| {\mathrm{e}^{\mathrm{i} t \Delta_\Gamma} P_c} u_0 \|_{L^{\infty}(\Gamma)} \leq C \dfrac{1}{\sqrt{|t|}} \ \| u_0 \|_{L^1(\Gamma)}, \quad \text{for all } u_0 \in L^1(\Gamma) \cap L^2(\Gamma) \text{ and } t \neq 0,
\end{equation}
where $P_c=I-P_p $, and $P_p$ is the orthogonal projection onto the linear span of the eigenfunctions, in $L^2(\Gamma)$.
\end{enumerate}
In subsections 2.4 and 2.5 we discuss general and concrete situations when hypothesis \ref{H1} is fulfilled.

\begin{defi}
Following \cite{debouarddebussche}, an exponent pair $(r,p)$ is called admissible if $r=\infty, p=2$ or $2 \leq r,p<\infty$ and $\frac{2}{r}+\frac{1}{p}>\frac{1}{2}$.
\end{defi}

In the sequel, in order to lighten the notations, we shall often write $L^{r}_{\omega} L^{r}_{[s,s+T]} L^{p}_x$ instead of $L^{r}(\Omega; L^{r}([s,s+T]; L^{p}(\Gamma)))$.

Extending \cite[Propositions 3.10 and 3.11]{debouarddebussche}, we get the following Strichartz estimates.
\begin{thm}\label{stochastic strichartz}
Assume that \ref{H1} is satisfied. 
Let $T>0$, $s \geq 0$, and $(r,p)$ an admissible exponent pair. 
Then\\

(i) 

\vspace{-30pt} 
\begin{equation*}
 \displaystyle{\big\|S_{\beta}(\cdot,s)X_s \big\|_{L^{r}_{\omega} L^{r}_{[s,s+T]} L^{p}_x} \leq c_{r,p} \; T^{\beta/2} \; \|X_s \|_{L^{r}_{\omega} L^{2}_x}}, \ \text{ with } \beta=\tfrac{2}{r}-\tfrac{1}{2} (\tfrac{1}{2} - \tfrac{1}{p}) ;
\end{equation*}

(ii) Let $(\gamma, \delta)$ be another admissible pair such that $\frac{1}{\gamma}=\frac{1-\lambda}{r}, \frac{1}{\delta}=\frac{\lambda}{2}+\frac{1-\lambda}{p} $ for some $\lambda \in [0,1]$. 

\medskip

\hspace{30pt} (ii.1) If $\max \{ \rho , \rho'\}  \leq r$, $\overline{\beta}=\beta (1-\frac{\lambda}{2}) $, and $f$ is predictable, then
\begin{equation*}
\qquad \qquad \qquad  \bigg\| \int_{s}^{\cdot} S_{\beta}(\cdot, \sigma) f(\sigma) d \sigma  \bigg\|_{L^{\rho}_{\omega} L^{r}_{[s,s+T]} L^{p}_x} \leq c_{r,p,\gamma,\delta,\rho} \;T^{\overline{\beta}} \; \big\| f \big\|_{L^{\rho}_{\omega} L^{\gamma'}_{[s,s+T]} L^{\delta'}_x};
\end{equation*}

\hspace{30pt} (ii.2) If $r' \leq \rho \leq r$, then
\begin{equation*}
\qquad \qquad \qquad  \bigg\| \int_{s}^{\cdot} S_{\beta}(\cdot, \sigma) f(\sigma) d \sigma  \bigg\|_{L^{ \rho}_{\omega} L^{\gamma}_{[s,s+T]} L^{\delta}_x} \leq c_{r,p,\gamma,\delta,\rho} \; T^{\overline{\beta}} \; \big\| f \big\|_{L^{\rho}_{\omega} L^{r'}_{[s,s+T]} L^{p'}_x},
\end{equation*}
\end{thm}


\subsection{Well-posedeness of equation (\ref{mild}) on $L^{2}(\Gamma)$} \label{nonlinear equation}

As already mentioned, the idea to solve \eqref{mild} on $L^2(\Gamma)$ is to apply the Banach fixed point theorem on some convenient space, based on the estimates obtained in Theorem \ref{stochastic strichartz}. However, looking at these estimates, one can notice that the smoothing effect is present in space-time, but not in $\Omega$.
For this reason, as in \cite{debouarddebussche}, we first need to consider a truncation in the $L^r_{[s,s+t]}L^p_x$-spaces, as follows:
let $\theta \in C_0^{\infty}(\mathbb{R})$ such that $\theta=1$ on $[0,1]$ and $\theta=0$ on $[2, \infty)$. 
For $X \in L^r_{\rm loc}([s,\infty),L^p_x)$  a.s., $R \geq 1$ and $t \geq 0$, we set
\begin{equation} \label{definition theta}
\theta_R^s(X)(t):=\theta\bigg(\frac{|X|_{L^r_{[s,s+t]}L^p_x}}{R}\bigg).
\end{equation}

\noindent For $s=0$, we set $\theta_R:=\theta^0_R$.\\

\noindent We consider the following truncated version of \eqref{linear stochastic}
\begin{equation} \label{truncated}
\left\{
\begin{array}{lll}
                 dX^R=-\frac{1}{2}\Delta_\Gamma^2 X^R \; dt + \mathrm{i} \Delta_\Gamma X^R \; d \beta(t) + \mathrm{i} \theta_R(X^R)|X^R|^{2\sigma}X^R \; dt, & t\in[0,T]\\ [5pt]
                  X^R(0) =X_0\\
         \end{array}
         \right.
\end{equation}
or in mild form:
\begin{align} \label{mild truncated}
X^R(t)=S_{\beta}(t,0)X_0+\mathrm{i}\int_0^t S_{\beta}(t,s)\theta_R(X^R)(s)|X^R(s)|^{2\sigma}X^R(s) \; ds, \quad t\in [0,T]\; a.s.
\end{align}

Throughout, by $L^r_{\mathcal{P}}(\Omega \times [0,T];L^{p}_x)$ we denote the predictable processes from $L^r(\Omega \times [0,T];L^{p}_x)$.

\begin{thm} \label{solution truncated}
Assume that \ref{H1} is satisfied and let $\sigma<2$, $p=2\sigma+2$ and $(r,p)$ be an admissible exponent pair. 
Then for any $\mathcal{F}_0$-measurable $X_0\in L^r_\omega L^2_x$ there exists a unique solution $X^R$ to \eqref{mild truncated} such that $X^R \in L^r_{\mathcal{P}}(\Omega \times [0,T];L^{p}_x)$ for any $T>0$, and has paths in $C(\mathbb{R}_+,L^2_x)$ a.s. 
Moreover, a.s. $|X^R(t)|_{L^2_x}=|X_0|_{L^2_x}, \; t\geq 0$  and $X^{R}\in L^{\rho}_{[0,T]}L^{q}_x$ a.s. for any $T>0$, $\rho \leq r$ and $(\rho,q)$ admissible.
\end{thm}

\begin{rem} \label{rem 2.9}
In \cite{debouarddebussche}, where $\Gamma = \mathbb{R}$, the proof of Theorem \ref{solution truncated} is based on a regularization of the solution to \eqref{truncated} obtained at a first stage by a fixed point argument on $L^r_{\mathcal{P}}(\Omega \times [0,T];L^{p}_x)$, using a cutoff in the Fourier space.
Since such a regularization cannot be performed on $\Gamma$, we have to use different arguments for the proof, which turn out to be simpler and more general; see Section 3.
\end{rem}

Based on Theorem \ref{solution truncated}, the arguments from \cite{debouarddebussche}, Section 5, work without any change to get the $L^{2}(\Gamma)$-well-posedness of \eqref{mild}.
Although we resume only to the statement (see Theorem \ref{wellposednessL2} below) and skip its proof, let us briefly explain how it can be worked out.
First of all, uniqueness follows by Theorem \ref{solution truncated}.
Then, using again Theorem \ref{solution truncated}, let $X^{R}_n$, $n\geq 0$ be the global solutions to \eqref{mild truncated}, obtained recursively for initial data $X_{\tau_R^{n}}$ where $\tau_R^{0}=0$ and
\begin{equation*}
\tau_R^{n}:=\inf\{t\geq0 : |X^{R}_n|_{L^{r}_{[0, t]}L^{p}_x}\geq R \} \;\; \mbox{ for all } n\geq 0.
\end{equation*}
By superposing $X_{\tau_R^{n}}1_{[0,\tau_R^{n}]}$ for all $n\geq 0$, we get a strong Markov solution to \eqref{mild} on $[0,\tau)$ where $\tau=\sum\limits_{n=0}^{\infty}\tau_R^{n}$.
To make sure that $\tau =\infty$ a.s., it is sufficient to show that there exists $\varepsilon>0$ s.t. $\lim\limits_{N}\lim\limits_{M}\mathbb{P}(\tau_R^{n}\leq \varepsilon \mbox{ for all } N\leq n\leq M )=0$.
But $\mathbb{P}(\tau_R^{n}\leq \varepsilon \mbox{ for all } N\leq n\leq M)=\mathbb{E}\big\{[\prod\limits_{n=N}^{M-1} 1_{\tau_R^{n}\leq \varepsilon}]\; \mathbb{P}(\tau_R^{M}\leq \varepsilon | \mathcal{F}_{\tau_R^{0}+\cdots+\tau_R^{M-1}})\big\}$.
Hence, it is sufficient to show that there exist $\varepsilon, R >0$ s.t. $\mathbb{P}(\tau_R^{M}\leq \varepsilon | \mathcal{F}_{\tau_R^{0}+\cdots+\tau_R^{M-1}})\leq \frac{1}{2}$ for all $M$, and this can indeed be obtained by the Strichartz estimates in Theorem \ref{stochastic strichartz}, the conservation of the $L^{2}$-norm obtained in Theorem \ref{solution truncated}, and the strong Markov property; for more details see \cite[Lemma 5.1]{debouarddebussche} and the discussion right after.

Consequently, we obtain:
\begin{thm} \label{wellposednessL2}
Assume that \ref{H1} is satisfied, and let $\sigma<2$ and $r$ such that $2\sigma+2\leq r < \frac{4(\sigma+1)}{\sigma}$. 
Then for any $X_0\in L^2(\Gamma)$, there exists a unique solution to \eqref{mild} which has paths in $C([0,\infty),L^2(\Gamma)) \, a.s.$.
In addition, $ |X(t)|_{L^2_x}=|X_0|_{L^2_x}, \; t\geq 0$ and $X\in L^{\rho}_{[0,T]}L^{q}_x$ a.s. for any $T>0$, $\rho \leq r$ and $(\rho,q)$ admissible.
\end{thm}

\subsection{Well-posedeness of \eqref{approximate} for truncated nonlinearities and convergence when $\varepsilon \to 0$} \label{convergence of approximate solutions}

In this section our aim is to extend the results from \cite{Ma06} to quantum graphs. More precisely, let $\theta$ be the function from the beginning of Subsection \ref{nonlinear equation}, $\sigma \geq \frac{1}{2}, R > 0$ and consider the function $F: \mathbb{C} \rightarrow \mathbb{C}$ given by
\begin{equation*}
F(x):= |x|^{2 \sigma} x \theta(|x|^2/R),
\end{equation*}
which is from $C^1_c $ with Lipschitz derivative. Then, the equations we deal with in this subsection are the truncated versions of \eqref{approximate}, namely

\begin{equation} \label{eq 2.13}
\begin{cases}
\mathrm{i} \dfrac{du}{dt} + \dfrac{1}{\varepsilon} m\bigg(\dfrac{t}{\varepsilon^2}\bigg) \Delta_{\Gamma} u + F(u) =0, & x \in \Gamma, \ t>0\\
u(0)=u_0, & x \in \Gamma
\end{cases},
\end{equation}
with the corresponding limiting equation
\begin{equation} \label{eq 2.14}
\begin{cases}
u dt = \mathrm{i} \Delta_{\Gamma}u d\beta - \dfrac{1}{2} \Delta^2_\Gamma u dt + \mathrm{i} F(u) dt, & x \in \Gamma, \ t>0\\
u(0)=u_0, & x \in \Gamma
\end{cases},
\end{equation}

We recall that we perform such a truncation because there are no Strichartz estimates available for general varying dispersion as for the white noise case from Subsection \ref{linear equation and strichartz}; as a consequence, \eqref{approximate} will be solved only locally. 

First, we prove well-posedness in $C([0,T],L^2(\Gamma))$ and in $C([0,T],D(\mathcal{E}))$ for the equation in mild form
\begin{equation} \label{mildform}
u_n(t)=S_{n}(t,0) u_0 + \mathrm{i} \int_0^t S_{n}(t,s) F(u_n)(s) \, ds,
\end{equation}

\noindent for any $n \in C([0,T],\mathbb{R})$, where $S_n(t,s):=\mathrm{e}^{\mathrm{i} [n(t)-n(s)]\Delta_{\Gamma}}$.

We need to consider the following stability of $D(\mathcal{E})$ under nonlinearity:
\begin{enumerate}
\item[\mylabel{H2}{\bf (H.2)}]
$
\!
\begin{aligned}[t]
\text{If } u \in D(\mathcal{E}), \text{ then } F(u) \in D(\mathcal{E}).
\end{aligned}
$ 
\end{enumerate}

\begin{thm}\label{well-posedness}
Let $T>0$ and $n \in C([0,T],\mathbb{R})$. 
Then, for any initial data $u_0 \in L^2(\Gamma)$, there exists a unique solution $u_n \in C([0,T],L^2(\Gamma))$ to \eqref{mildform}.

Moreover, if $ u_0 \in D(\mathcal{E})$ and $F$ satisfies \ref{H2}, then $u_n \in C([0,T],D(\mathcal{E}))$.

\end{thm}

\begin{rem}
In particular, Theorem \ref{well-posedness} applies to \eqref{eq 2.13} and \eqref{eq 2.14} for $n(t)=\int_0^t \frac{1}{\varepsilon} m \big( \frac{s}{\varepsilon^2} \big) \, ds$, and for $n(t)=\beta(t)$, respectively, $t \geq 0$.
\end{rem}

The main result of this section is the following. 

\begin{thm}\label{convergence} If \ref{H0} and \ref{H2} hold, then the mild solutions $(u_\varepsilon)_{\varepsilon>0}$ of (\ref{eq 2.13}) converge in law to the mild solution $u$ of (\ref{eq 2.14}) when $\varepsilon \to 0$, on $C([0,T],D(\mathcal{E}))$.
\end{thm}

\subsection{Well-posedness in $D(\mathcal{E}_{\star})$ for full nonlinearity and convergence of the approximate solutions: the case of star-graphs} \label{well posedness star-graph}

Concerning the well-posedeness of \eqref{mild} in the energy domain $D(\mathcal{E})$ given by \eqref{formdomain}, let us mention once again that the dispersion destroys the conservation in time of the energy $\mathcal{E}(X(t),X(t))-\frac{1}{2 \sigma+2} \| X(t) \|^{2 \sigma+2}_{L^{2 \sigma+2}(\Gamma)}$, hence an approach as in e.g. \cite[Theorem C]{grecuignatjphysa} (more precisely, see Proposition 3.7) or \cite{AdNo09} is not suitable.  
On $\mathbb{R}$, the well-posedeness in $W^{1,2}(\mathbb{R})$ has been obtained in \cite{debouarddebussche} based on the Strichartz estimates and the fact that the first derivate operator commutes with the deterministic Schr\"odinger group, i.e. $\frac{d}{dx}e^{it\Delta}=e^{it\Delta}\frac{d}{dx}$.  
However, the situation changes drastically in our case because such a commuting property simply does not hold.
To overcome this problem, the strategy is to employ spectral arguments and explicit heat kernels formulas in order to be able to partially commute $\frac{d}{dx}$ with $e^{it\Delta_\Gamma}$, with the price of a reminder part which hopefully is also smoothing.
Since it may be too difficult to get such formulas for quantum graphs in general, we restrict our analysis to the case of star graphs, with the emphasis that our strategy is a general one and could be applied to other types of graphs. So, our main concern in this section is to prove well-posedness of equation \eqref{mild} in the energy domain $D(\mathcal{E}_{\star})$ described in \eqref{formdomain}, where $\mathcal{E}_{\star}$ denotes the form associated to a self-adjoint Hamiltonian on a star-graph, as well as the convergence of the solutions of \eqref{approximate}, when $\varepsilon \to 0$. We recall that a star-graph $\Gamma_{\star}$, is a metric graph that consists of a finite number $n \in \mathbb{N}^*$ of infinite length edges attached to a single common vertex, with each edge being identified with a copy of the positive real axis, $[0, \infty)$. 
\begin{figure}[H]
\centering
\includegraphics[scale=0.11]{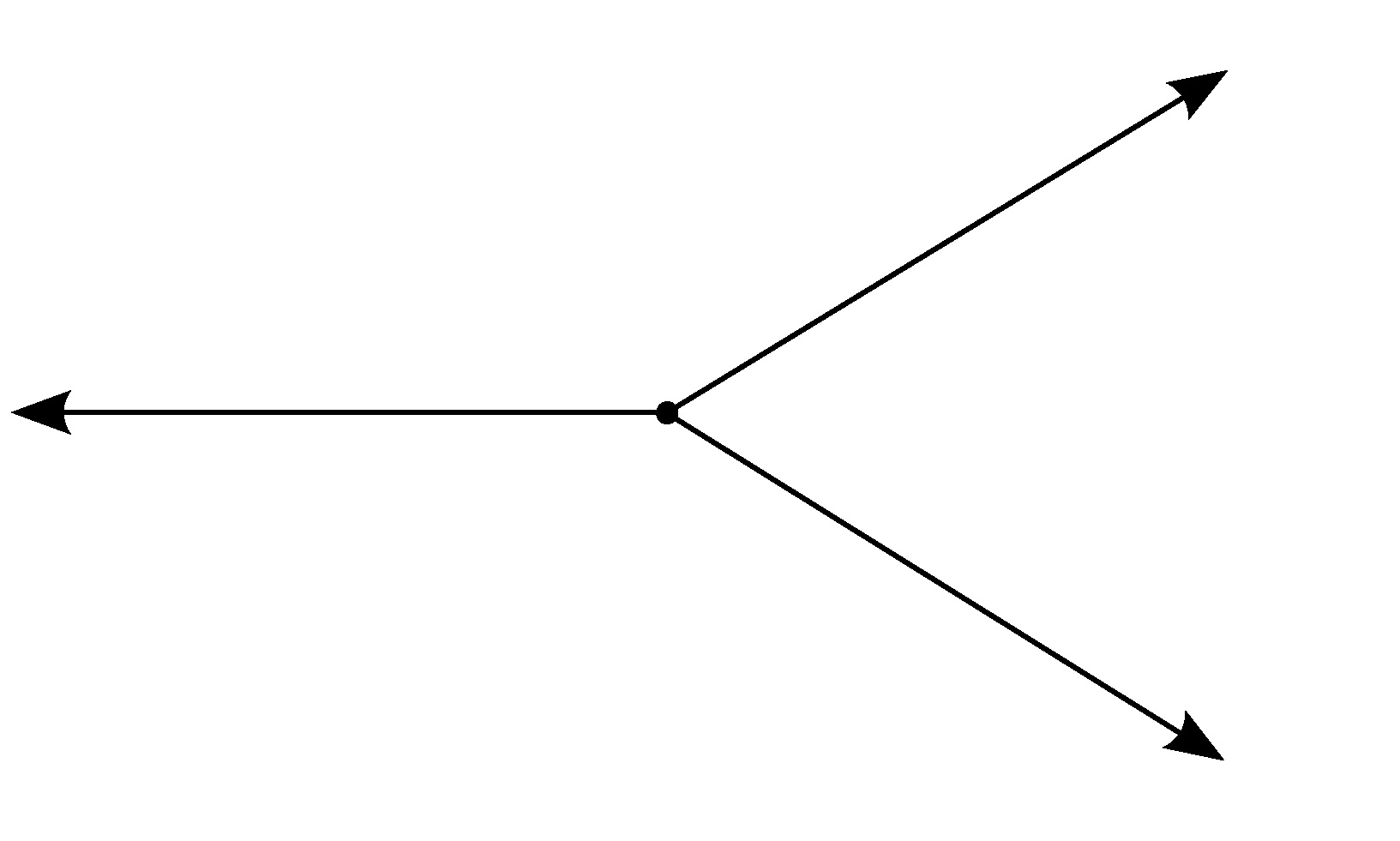} \caption{\small{Star-graph, $\Gamma_{\star}$. }}\label{stargraph} 
\end{figure}
We denote by $H_{\star}:=-\Delta_{\Gamma_{\star}}(A,B)$ the Hamiltonian on $\Gamma_{\star}$, with $A$ and $B$ satisfying the hypotheses of Theorem \ref{self-adjointconditions}. 

\begin{prop} \label{prop 2.13} Condition \ref{H1} is fulfilled for $H_{\star}$. In particular, Theorem \ref{wellposednessL2} applies.
\end{prop}

\begin{proof}
The estimate \eqref{H1eq} is one of the main results in \cite{grecuignatjphysa}, more precisely, Theorem A. The second part of \ref{H1} follows by \cite[Identity (3.1) and Theorem 3.7]{laplacians}, according to which there are no positive eigenvalues of $H_{\star}$ and the number of negative ones is finite counting their multiplicities and equals precisely the number of positive eigenvalues of $AB^{\dag}$, denoted further by $n_+(AB^{\dag})$.
\end{proof}

In fact, the work from \cite{grecuignatjphysa} extends to general coupling conditions the dispersive properties obtained in \cite{adami} for the following three particular couplings (see \cite{berkolaiko} for details and physical meaning):

\noindent 1. The {\it Kirchhoff Hamiltonian} $\Delta^K_{\Gamma}$ with domain
\begin{equation*}\label{kirchhoff}
\mathcal{D}(\Delta_{\Gamma}) = \big\{ \psi \in H^2(\Gamma): \ \psi_i(0)=\psi_j(0), \ 1 \leq i,j \leq n, \ \sum_{j=1}^{n} \psi'_j(0_+)=0  \big\};
\end{equation*}
\noindent 2. The {\it $\delta$ (Delta) Hamiltonian} $\Delta^{\delta}_{\Gamma}$ with domain
\begin{equation*}\label{delta}
\mathcal{D}(\Delta^{\delta}_{\Gamma}) = \big\{ \psi \in H^2(\Gamma): \ \psi_i(0)=\psi_j(0), \ 1 \leq i,j \leq n, \ \sum_{j=1}^{n} \psi'_j(0_+)= \alpha \psi_1 (0) \big\}, \alpha \in \mathbb{R};
\end{equation*}
\noindent 3. The {\it $\delta'$ (Delta-prime) Hamiltonian} $\Delta^{\delta'}_{\Gamma}$ with domain
\begin{equation*}\label{delta-prime}
\mathcal{D}(\Delta^{\delta'}_{\Gamma}) = \big\{ \psi \in H^2(\Gamma): \ \psi'_i(0_+)=\psi'_j(0_+), \ 1 \leq i,j \leq n, \ \sum_{j=1}^{n} \psi_j(0)= \beta \psi_1 (0) \big\}, \beta \in \mathbb{R}.
\end{equation*} 

\noindent To be more precise, in \cite{adami} $\alpha$ and $\beta$ are assumed to be strictly positive, and in these cases, \ref{H1} is satisfied for $ \mathrm{e}^{\mathrm{i} t \Delta} $ instead of $\mathrm{e}^{\mathrm{i} t \Delta} P_c$.

\begin{prop}
For Kirchhoff and $\delta-$couplings, condition \ref{H2} is also satisfied. In particular, Theorems \ref{well-posedness} and \ref{convergence} also apply.
\end{prop}

\begin{proof}
It is clear since $D(\mathcal{E}_{\star}):= \{ \psi \in W^{1,2}(\Gamma) \, : \, \psi_i(0)=\psi_j(0), \ 1 \leq i,j \leq n  \} $ in both cases.
\end{proof}

\noindent {\bf Coupling conditions with no Robin part.}

\noindent The following commuting formula will be crucially used and it could be itself of general interest.

\begin{prop} \label{derivativesemigroup} If $P_R=0$, i.e. there is no Robin coupling, and $v \in D(\mathcal{E}_{\star})$, then
\begin{equation} \label{formuladerivativesemigroup}
\partial_x (\mathrm{e}^{- \mathrm{i} t H_{\star}} v(x))= \mathrm{e}^{-\mathrm{i} t H_{\star}} P_c v' (x) + 2 \mathrm{e}^{ \mathrm{i} t \Delta_{\mathbb{R}}} \widetilde{v'}(x) + \sum_{l=1}^{n_+(AB^{\dag})} \langle v,\varphi_l \rangle \mathrm{e}^{-\mathrm{i}t \lambda_l} \varphi_l'(x),
\end{equation}

\noindent where $\mathrm{e}^{\mathrm{i} t \Delta_{\mathbb{R}}}$ is the unitary group generated by the free Schr\"{o}dinger operator on $\mathbb{R}$, $\{ (\lambda_l,\varphi_l) \}_l$ are the eigencouples of $H_{\star}$ and $\widetilde{v'}=(v'_1,v'_2,\dots,v'_n)^T$, with
\begin{equation}\label{extended}
\widetilde{v'_j}(x)=
\begin{cases}
v'_j(x), & x \in [0, \infty)\\
0, & x \in (-\infty,0)
\end{cases}.
\end{equation}

\end{prop}

We have the following well-posedness result.

\begin{thm}\label{thm 2.16}
Let $\sigma<2$, $p=2\sigma + 2$ and $(r,p)$ an admissible exponent pair. If \ref{H2} holds and $P_R =0$, then, for any $\mathcal{F}_0$-measurable $X_0\in L^r_\omega D(\mathcal{E}_{\star})$ there exists a unique solution $X$ to \eqref{mild} with $a.s$ paths in $C(\mathbb{R_+},D(\mathcal{E_{\star}}))$.
\end{thm}

Finally, let us consider the approximate equations \eqref{approximate}, for which we emphasize that merely local solutions can be obtained, due to Theorem \ref{well-posedness}.
Nevertheless, as in \cite{DeTs11}, the global well-posedness in the energy domain of the limiting equation \eqref{eq 2.6} allows us to prove that the local solutions of \eqref{approximate} converge to the desired limit.
More precisely, we have:

\begin{thm} \label{convergence3}
Let $\sigma \in [1/2, 2)$ and assume that \ref{H2} holds and $P_R=0$. If $X_0 \in D(\mathcal{E}_\star)$, then for every $\varepsilon >0$ there exists a unique (mild) solution $X_\varepsilon$ to \eqref{approximate} with continuous paths in $D(\mathcal{E}_\star)$, which is defined on a random time interval $[0, \tau_\varepsilon )$.
Moreover, for any $T>0$
\begin{equation*}
\begin{aligned}
&\lim\limits_{\varepsilon \to 0} \mathbb{P}([\tau_\varepsilon<T])=0\\
\end{aligned}
\end{equation*}
and the process $X_\varepsilon \mathbbm{1}_{[0, \tau_\varepsilon)}$ converges in law to the solution $X$ of \eqref{mild} when $\varepsilon \to 0$, on $C([0,T];D(\mathcal{E}_\star))$.
\end{thm}

\begin{rem}
In the cases of Kirchhoff and Dirichlet boundary conditions, \ref{H2} holds and $P_R=0$, hence Theorems \ref{thm 2.16} and \ref{convergence3} apply.
\end{rem}

\noindent {\bf A case of non-zero Robin part: $\delta-$type condition.} If $P_R \neq 0$, we do not know if $\mathrm{e}^{\mathrm{i} t \Delta_{\Gamma}}$ has smoothing effect on $L^r_t W^{1,p}_x$ spaces in general. Nevertheless, we have at least one example of interest where $P_R=0$, namely $\delta-$ coupling conditions, for which a formula similar to \eqref{formuladerivativesemigroup} holds, and it ensures the desired smoothing effect.

\begin{prop} \label{derivativesemigroupdelta}
For $\delta-$coupling conditions, the following holds for all $v \in D(\mathcal{E}_{\star})$:
\begin{equation}
\partial_x (\mathrm{e}^{- \mathrm{i} t H_{\star}} v(x))= \mathrm{e}^{-\mathrm{i} t H_{\star}} P_c v' (x) + 2 \mathrm{e}^{ \mathrm{i} t \Delta_{\mathbb{R}}} \widetilde{v'}(x) + \sum_{l=1}^{n_+(AB^{\dag})} \langle v,\varphi_l \rangle \mathrm{e}^{-\mathrm{i}t \lambda_l} \varphi_l'(x) + \dfrac{2 \alpha v_1(0)}{n} \mathrm{e}^{\mathrm{i}t \Delta_{\mathbb{R}}} \varphi (x),
\end{equation}

\noindent with $\{ (\lambda_l,\varphi_l) \}_l$ the eigencouples of $H_{\star}$, $\psi(y):= \mathrm{e}^{\frac{\alpha}{3} y} \mathbbm{1}_{(-\infty,0]}(y)$ and $\widetilde{v'}$ as in \eqref{extended}.
\end{prop}

Since $D(\mathcal{E})$ is continuously embedded in $L^{\infty}(\Gamma)$, we can bound $|u_{0,1}(0)|$ with $ \| u_0 \|_{D(\mathcal{E})}$,and the proofs of Theorems \ref{thm 2.16} and \ref{convergence3} can be easily adjusted to get:

\begin{thm} \label{thm 2.19}
The conclusions of Theorems \ref{thm 2.16} and \ref{convergence3} are valid for $\delta-$coupling conditions.
\end{thm}

\subsection{Further examples}
First of all, we would like to emphasize that the results from Subsection \ref{nonlinear equation} hold for self-adjoint operators which satisfy \ref{H1}, on any $L^2-$space, i.e. the metric graph structure was not crucial for those results. For example, in \cite[Theorem 1.2]{ignatzuazua} it was shown that \ref{H1} is satisfied with $P_{c}=Id$ for $H=\frac{\partial}{\partial x} (a(x) \frac{\partial}{\partial x})$ on $L^2(\mathbb{R})$, with $a \in BV(\mathbb{R})$, bounded, such that $Var(log(a))<2 \pi$. In particular, Theorem \ref{wellposednessL2} applies to prove well-posedness on $L^2(\mathbb{R})$ for equation \eqref{mild} with $\Delta_{\Gamma}$ replaced by $H$. 

In the sequel, we present several examples of coupling conditions which induce self-adjoint extensions of the Laplacian $\Delta_{\Gamma}$ on different types of metric graphs $\Gamma$, for which the $L^1-L^{\infty}$ dispersive estimate \ref{H1} and stability condition \ref{H2} hold, other than star-graphs, the latter being already extensively studied in Subsection \ref{well posedness star-graph}.

\subsubsection{Simple graphs with internal edges}

We include here the case of the Schr\"{o}dinger group on the real line with several point defects, which can be regarded as simple graphs with a finite number $n \in \mathbb{N}^*$ of edges, with particular self-adjoint coupling conditions at each vertex.
\begin{figure}[H]
\centering
\includegraphics[scale=0.17]{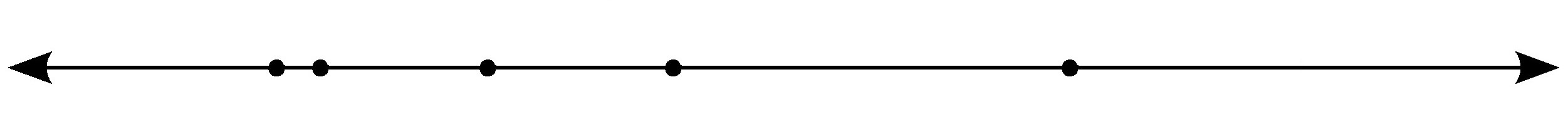}
\caption{\small{Simple graph with $n$ edges: $n-2$ internal ones and $2$ external ones.}}\label{stargraph} 
\end{figure}

In \cite{AdNo09}, the real line setting with a single point defect was considered. More precisely, dispersive estimates \ref{H1} are fulfilled in the case of all self-adjoint extensions, which can be described in one of the following ways:
\begin{equation*}
\mathcal{D}(H_U) = \big\{ \psi \in H^2(\mathbb{R} \setminus \{ 0 \}): \ \left[ \begin{array}{c} \psi(0_+) \\ \psi'(0_+) \end{array} \right] = U  \left[ \begin{array}{c} \psi(0_-) \\ \psi'(0_-) \end{array} \right] \big\},
\end{equation*}
\noindent where $U= \omega \begin{bmatrix} a & b \\ c & d \end{bmatrix}$ is given, with $|\omega|=1$ and $ad-bc=1$, or
\begin{equation*}
\mathcal{D}(H_{p,q}) = \big\{ \psi \in H^2(\mathbb{R} \setminus \{ 0 \}): \ \psi(0_+)=p \psi'(0_+), \psi(0_-)=q \psi'(0_-) \big\},
\end{equation*}
\noindent with given $p,q \in \mathbb{R} \cup \{ \infty \}$. Moreover, by \cite[Proposition 2.1]{AdNo09}, we can immediately deduce that the form domain stability condition \ref{H2} is satisfied for all $H_{p,q}$. In the case of $H_U$, if $b \neq 0$, than \ref{H2} is fulfilled. Otherwise, \ref{H2} is satisfied provided that $\omega \alpha = 1 $. Note that these can also be viewed in the framework of star-shaped graphs with two infinite length edges attached to a common vertex.

In \cite{kovariksacchetti2010}, the case of two symmetric Delta Dirac potentials placed at points $\pm a \in \mathbb{R}$ was considered. The Hamiltonian $\Delta_{\alpha}$ taken into consideration has domain:
\begin{equation*} \label{symmetricdelta}
\begin{aligned}
\mathcal{D}(\Delta_{\alpha}) = \big\{ \psi \in H^2(\mathbb{R} \setminus \{ \pm a \}): & \ \psi(\pm a +0) = \psi(\pm a-0), \\ 
& \psi'(\pm a +0) - \psi'(\pm a -0) = \alpha \psi(\pm a +0), \ \alpha \in \mathbb{R} \big\}.
\end{aligned}
\end{equation*}

\noindent They proved that the corresponding group $\mathrm{e}^{\mathrm{i} t \Delta_{\alpha}} P_c$ satisfies \ref{H1} provided that $a \alpha \neq -1$. We remark that for $\alpha \geq 0$ the discrete spectrum of $-\Delta_{\alpha}$ is empty and thus the dispersive estimate \ref{H1} holds true for $\mathrm{e}^{\mathrm{i} t \Delta_{\alpha}}$. Furthermore, \ref{H2} holds since the form domain is $H^1(\mathbb{R}_-) \oplus H^1(\mathbb{R}_+)$ with continuity at the points $\pm a$.

The dispersive results in \cite{kovariksacchetti2010} were extended in \cite{banicaignat2014} to several Dirac Deltas located at finitely many points $\{ x_j\}_{j=1}^{p}$, $p \in \mathbb{N}^*$, on the real line, with the associated strengths $\{\alpha_j \}_{j=1}^{p}$. More precisely, the domain of this Hamiltonian is 
\begin{equation*}
\begin{aligned}
\mathcal{D}(\Delta_{\alpha,p}) = \big\{ \psi \in H^2(\mathbb{R} \setminus & \{ \{x_j\}_{j=1}^{p} \}): \ \psi(x_j +0) = \psi(x_j-0), \\
& \psi'(x_j +0) - \psi'(x_j-0) = \alpha_j \psi(x_j +0), \ \alpha_j \in \mathbb{R}, \ j=1, \cdots, p \big\}.
\end{aligned}
\end{equation*}

\noindent The strengths $\{ \alpha_j \}_{j=1}^{p}$ and the points $\{x_j \}_{j=1}^{p}$ are assumed to satisfy some technical condition which excludes only a few explicit situations. We mention that in the case of positive strengths $\alpha_j > 0$, $j=1 \cdots p$, this condition is fulfilled and, moreover, \ref{H1} is satisfied by $\mathrm{e}^{\mathrm{i}t \Delta_{\alpha,p}}$. Clearly, \ref{H2} is again satisfied.

\subsubsection{Trees}

In this subsection we place ourselves in the framework of trees, more precisely, a particular case of regular trees (Figure \ref{ignattree}) and slightly more general ones (Figure \ref{banicaignattree}). A tree is a graph which has each two vertices connected by a single path of edges, and we say that the tree is {regular} if all the vertices of the same generation have equal number of descendants, and all edges from the same generation are of the same length (for more details, we refer to \cite{laplaceonregulartrees, mugnolo}).
\begin{figure}[H]
  \centering
  \begin{minipage}[b]{0.45\textwidth}
  \centering
    \includegraphics[width=4.5cm]{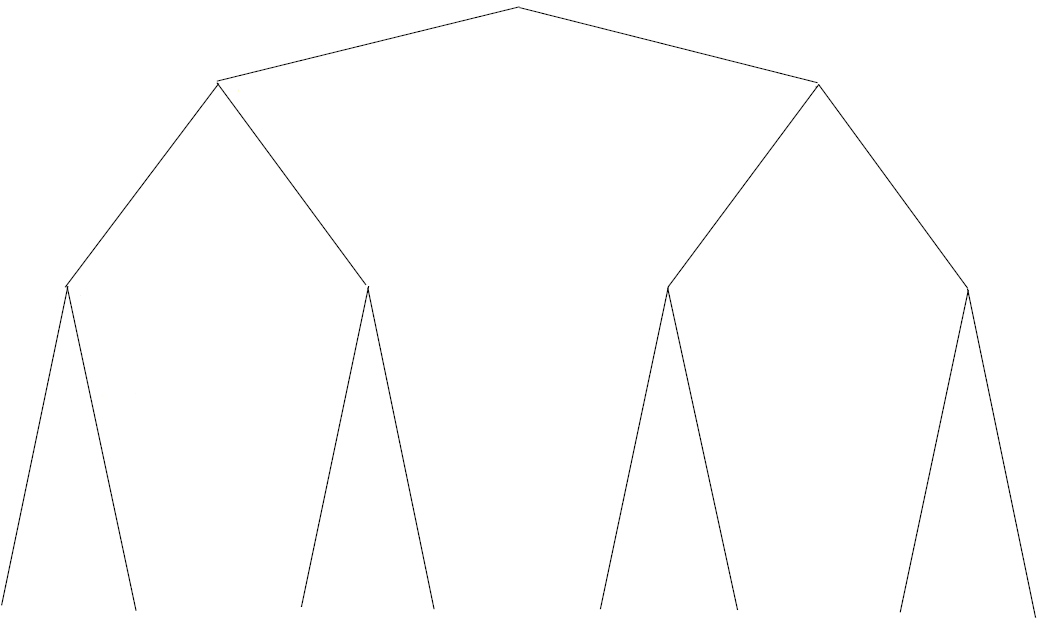}
    \caption{\small{Regular tree, $\Gamma_r$, with the last generation formed by infinite edges.}}\label{ignattree}
  \end{minipage}
  \hfill
  \begin{minipage}[b]{0.45\textwidth}
  \centering
    \includegraphics[width=4cm]{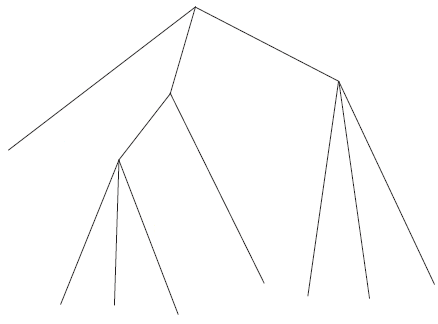}
    \caption{\small{A more general tree, $\Gamma_t$, with vertices of degree greater or equal three.}}\label{banicaignattree}
  \end{minipage}
\end{figure}

In both cases, $\Gamma_r$ and $\Gamma_t$, we denote by $E_v$ the set of edges adjacent to the vertex $v$. For each vertex $v$ and $E_v$, we set, if the edge $e$ is of finite length,
\begin{equation*}
i(v,e)=
\left\{
\begin{array}{lll}
0, & v=\partial^-(e)\\
l_e, & v=\partial^+(e)
\end{array}, 
\right.
\end{equation*}
otherwise $ i(v,e)=0.$
The normal derivative of the restriction of $\psi$ on the edge $e \in E_v$ evaluated at the endpoints is
\begin{equation*}
\dfrac{\partial \psi_e}{\partial n_e} (i(v,e))=
\left\{
\begin{array}{lll}
-\psi'_e(0_+), & i(v,e)=0\\
\psi'_e({l_e}_-), & i(v,e)=l_e\\
\end{array}.
\right.
\end{equation*}
Consider now the Laplacian with Kirchhoff coupling conditions at each vertex of the tree
\begin{equation*}
\mathcal{D}(\Delta_{\Gamma}) = \big\{ \psi \in H^2(\Gamma): \ \psi_e(i(v,e))=\psi_{e'}(i(v,e')), e,e' \in E_v,  \sum_{e \in E_v} \dfrac{\partial \psi_e}{\partial n_e} (i(v,e))=0, \ v \in V \big\},
\end{equation*}
where $\Gamma=\Gamma_r$ or $\Gamma=\Gamma_t$. In \cite{ignatsiam} (for $\Gamma_r$) and in \cite{banica2011} (for $\Gamma_t$), it was shown that the dispersive estimate \ref{H1} holds. Moreover, \ref{H2} is satisfied in both cases, since the form domain consists of $H^1(\Gamma)$ functions with continuity at the vertices.

\section{Proofs of the main results}

\subsection{Proofs of results from Subsection \ref{linear equation and strichartz}}

\begin{proof}[\bf Proof of Proposition \ref{linear solution}] 
The fact that a.s. $(S_\beta(t,s)X_s)_{t\geq s}$ has paths in $C([s,\infty),D(\Delta_\Gamma^2))$ and $|X(t)|_{L^2}=|X(s)|_{L^2}, t\geq s$ follows directly from \eqref{stochastic semigroup} and the properties of the deterministic Schr\"{o}dinger group $(e^{it\Delta_\Gamma})_{t\in \mathbb{R}}$.
To show that $S_\beta$ is the unique solution to \eqref{linear stochastic}, we rely on the following well known spectral representation of $f(-\Delta_\Gamma)$ which holds for any continuous function $f:\mathbb{R}\rightarrow \mathbb{C}$:
\begin{equation}\label{spectral}
f(-\Delta_\Gamma)u_0= \lim\limits_n\lim\limits_{\varepsilon \to 0}\int_{-n}^{n}f(\lambda)[R_{\lambda+\mathrm{i}\varepsilon}-R_{\lambda-\mathrm{i}\varepsilon}]u_0\; d \lambda, \quad \forall u_0 \in D(f(\Delta_\Gamma)),
\end{equation}
where $R(z):=(-\Delta_{\Gamma}- z )^{-1}$, for $z$ in the resolvent set of $-\Delta_{\Gamma}$.
The limits are in $L^2(\Gamma)$, and their order cannot be reversed.

Let $X_s\in D(\Delta_\Gamma^2) \ \mathbb{P}\mbox{-.a.s.}$ 
By \eqref{spectral}, we have 
\begin{equation}\label{spectral for X}
S_\beta(t,s)X_s= \lim\limits_n\lim\limits_{\varepsilon \to 0}\int_{-n}^{n} {\mathrm{e}^{\mathrm{i}[\beta(t)-\beta(s)]\lambda}}[R_{\lambda+\mathrm{i}\varepsilon}-R_{\lambda-\mathrm{i}\varepsilon}]X_s \; d\lambda,
\end{equation}
hence $(S_\beta(t,s)X_s)_{t\geq s}$ is $(\mathcal{F}_t)_{t \geq s}$-adapted. 
Using It\^o formula  for $\mathrm{e}^{\mathrm{i}[\beta(\cdot)-\beta(s)]\lambda}$ in \eqref{spectral for X}, we easily arrive at
\begin{align*}
S_\beta(t,s)X_s &=X_s +\mathrm{i} \lim\limits_n\lim\limits_{\varepsilon \to 0}\int_{-n}^{n} \int_s^t {\lambda} {\mathrm{e}^{\mathrm{i}[\beta(r)-\beta(s)]\lambda}} [R_{\lambda+\mathrm{i}\varepsilon}-R_{\lambda-\mathrm{i}\varepsilon}]X_s \; d\beta(r) \; d\lambda \qquad \qquad \qquad \\
& \quad -{\frac{1}{2}} \lim\limits_n\lim\limits_{\varepsilon \to 0}\int_{-n}^{n}  \int_s^t  \lambda^2 {\mathrm{e}^{\mathrm{i}[\beta(r)-\beta(s)]\lambda}} [R_{\lambda+\mathrm{i}\varepsilon}-R_{\lambda-\mathrm{i}\varepsilon}]X_s \; dr \; d\lambda.
\end{align*}
By Fubini and dominated convergence theorems (classical and stochastic versions, see e.g. \cite{daprato}, \cite{liurockner}), but also \eqref{spectral}, we get a.s.
\begin{align*}
S_\beta(t,s)X_s & =X_s+\mathrm{i} \int_{s}^{t} \lim\limits_n\lim\limits_{\varepsilon \to 0} \int_{-n}^{n}  {\lambda} {\mathrm{e}^{\mathrm{i}[\beta(r)-\beta(s)]\lambda}} [R_{\lambda+\mathrm{i}\varepsilon}-R_{\lambda-\mathrm{i}\varepsilon}]X_s \; d\lambda \; d\beta(r) \\
& \ \ \qquad -{\frac{1}{2}} \int_{s}^{t}  \lim\limits_n\lim\limits_{\varepsilon \to 0} \int_{-n}^{n}  \lambda^2 {\mathrm{e}^{\mathrm{i}[\beta(r)-\beta(s)]\lambda}} [R_{\lambda+\mathrm{i}\varepsilon}-R_{\lambda-\mathrm{i}\varepsilon}]X_s \; d\lambda \; dr\\
&=X_s-{\frac{1}{2}} \int_{s}^{t} \Delta_{\Gamma}^2 X(r) \; dr + \mathrm{i} \int_{s}^{t} \Delta_{\Gamma} X(r) \; d\beta(r) \; \quad \mbox{ for all } t \geq s.
\end{align*}

Finally, if $(Z(t))_{t\geq 0}$ is another solution to \eqref{linear stochastic}, by similar spectral arguments and stochastic calculus as above, one can easily check that ${\mathrm{e}^{-\mathrm{i} [\beta(t)-\beta(s)] \Delta_\Gamma}} Z(t)= X_s \ \mathbb{P}\mbox{-a.s.}$ for all $t\geq 0$.
But this clearly completes the proof since both $Y$ and $S_\beta(\cdot,s)$ have continuous paths, and by the group property of $({\mathrm{e}^{\mathrm{i} t \Delta_\Gamma}})_{t \in \mathbb{R}}$ we obtain 
\begin{equation*}
Z(t)={\mathrm{e}^{\mathrm{i}[\beta(t)-\beta(s)]\Delta_\Gamma}}X_s=S_\beta(t,s) \ \mathbb{P}\mbox{-.a.s.}, \mbox{ for all } t \geq 0. \qedhere
\end{equation*}
\end{proof}

\begin{proof}[\bf Proof of Theorem \ref{stochastic strichartz}] 

Let us set for all $s$ and $t \in \mathbb{R}_+$, 
\begin{equation}
\begin{aligned}
S^c_{\beta}(t,s):=S_{\beta}(t,s)P_c \text{ and } S^p_{\beta}(t,s):=S_{\beta}(t,s)P_p.
\end{aligned}
\end{equation}

\noindent Clearly, it is sufficient to prove the desired estimates for $S^c_{\beta}$ and $S^p_{\beta}$ separately. Concerning $S^c_{\beta}$, note first that from \ref{H1} and the fact that the deterministic group $(S(t))_{t\in \mathbb{R}}$ is an isometry on $L^2(\Gamma)$, by Riesz-Thorin interpolation theorem we get
\begin{equation*}
\| S^c(t-s) X_s\|_{L^p(\Gamma)} \lesssim {|t-s|^{\frac{1}{p}-\frac{1}{2}}}\| X_s\|_{L^{p'}(\Gamma)}, \quad \text{for all } X_s\in L^{p'}(\Gamma),
\end{equation*}
\noindent hence,
\begin{equation*}
\|S^c_{\beta}(t,s)X_s\|_{L^p(\Gamma)} \lesssim {|\beta(t)-\beta(s)|^{\frac{1}{p}-\frac{1}{2}}} \|X_s\|_{L^{p'}(\Gamma)}, \quad \text{for all } X_s \in L^{p'}(\Gamma) \ \mathbb{P}\mbox{-a.s.}
\end{equation*}

\noindent Then, the proofs of Propositions 3.7, 3.10 and 3.11 in \cite{debouarddebussche} work without any change to get the estimates for $S^c_{\beta}$. In the case of $S^p_{\beta}$, for all $t \in \mathbb{R_+}$,
\begin{equation*}
S^p_{\beta}(t)X_s = \sum_{j=1}^{|\sigma_p(-\Delta_{\Gamma})|} \mathrm{e}^{-\mathrm{i} \beta(t) \lambda_j} \langle X_s,\varphi_j \rangle_{L^2(\Gamma)} \varphi_j,
\end{equation*}

\noindent where $\lambda_j \in \mathbb{R}$ are the eigenvalues of $-\Delta_{\Gamma}$ and $\varphi_j$ the corresponding eigenfunctions, which by \cite[Section 3]{laplacians}, belong to $L^{\alpha}(\Gamma)$, for all $1 \leq \alpha \leq \infty$. Let $(r,p)$ be an admissible pair. For every $\alpha \geq 1$, similarly to \cite[Proof of Corollary 1] {grecuignatjphysa} we get
\begin{equation*}
\| S^p_{\beta}(t) X_s \|_{L^p(\Gamma)} \leq  \Bigg( \sum_{j=1}^{|\sigma_p(-\Delta_{\Gamma})|} \| \varphi_j \|_{L^{\alpha'}(\Gamma)} \| \varphi_j \|_{L^p(\Gamma)} \Bigg) \| X_s \|_{L^{\alpha}(\Gamma)},
\end{equation*}
\begin{equation} \label{Pp Strichartz}
\| S^p_{\beta}(\cdot,s) X_s \|_{L^r_{\omega}L^r_{[s,s+T]}L^p_x} \leq   \Bigg(\sum_{j=1}^{|\sigma_p(-\Delta_{\Gamma})|} \| \varphi_j \|_{L^{\alpha'}(\Gamma)} \| \varphi_j \|_{L^p(\Gamma)} \Bigg) T^{1/r} \| X_s \|_{L^r_{\omega}L^{\alpha}_x},
\end{equation}
\begin{equation} \label{Pp Strichartz 2}
\Big\| \int_s^{\cdot} S^p_{\beta}(\cdot,\sigma) f (\sigma) \, d\sigma \Big\|_{L^r_{\omega}L^r_{[s,s+T]}L^p_x} \leq   \Bigg(\sum_{j=1}^{|\sigma_p(-\Delta_{\Gamma})|} \| \varphi_j \|_{L^{\alpha'}(\Gamma)} \| \varphi_j \|_{L^p(\Gamma)} \Bigg) T^{1/r} \| f \|_{L^r_{\omega}L^1_{[s,s+T]},L^{\alpha}_x},
\end{equation}

\noindent which conclude the desired estimates.
\end{proof}

\subsection{Proofs of results from Subsection \ref{nonlinear equation}}

\begin{proof}[\bf Proof of Theorem \ref{solution truncated}] Denoting by $\mathcal{T} X^R$ the right hand side of \eqref{mild truncated}, the latter becomes $X^R= \mathcal{T} X^R$. The proof which relies on the Strichartz type estimates obtained in Theorem \ref{stochastic strichartz}, will be split in four steps.

\noindent {\it Step 1.} $\mathcal{T}$ is a contraction on the "space-time-omega" balls:
\begin{equation*}
E(T,a):=\{X \; \mathcal{F}_t-adapted: \; \vertiii{X} := |X|_{L^r_\omega L^{\infty}_{[0,T]}L_x^2}+|X|_{L^r_{\mathcal{P}}(\Omega \times [0,T];L^{p}_x)} \leq a\}
\end{equation*}

\noindent for some convenient $a$ and $T$, where $T$ does not depend on the initial data so that this procedure can be iterated to obtain a global solution.

\noindent {\it Step 2}.  The solution $X^R$ from {\it Step 1} possesses {extra integrability} properties: for any $T>0$, $\rho\leq r$ and $(\rho,q)$ admissible,
\begin{equation*}
X^{R}\in L^{\rho}_{[0,T]}L^{q}_x \; \; \; \mathbb{P}\mbox{-a.s.}
\end{equation*}

\noindent {\it Step 3.} $X^R$ has a version with trajectories in
$C([0,T];L^2_x) \; \; \; \mathbb{P}\mbox{-a.s.}$.

\noindent {\it Step 4.} For any $t>0$, 
\begin{equation*}
|X^R(t)|_{L^2_x}=|X_0|_{L^2_x} \; \; \;  \mathbb{P}\mbox{-a.s.}
\end{equation*}

\noindent{\it Proof of Step 1.} Note that $E(T,a)$ is a complete metric space. First we find $T$ and $a$ such that $\mathcal{T}:E(T,a) \to E(T,a)$ is well-defined. To do this, for any $t\in[0,T]$, by the Strichartz type estimates from Theorem \ref{stochastic strichartz}, we have
\begin{align*}
| \mathcal{T} X^R|_{L^r_\omega L^{\infty}_{[0,T]}L_x^2}  &\leq |X_0|_{L^r_\omega L_x^2}+\bigg| \int_0^t S_{\beta}(t,s)\theta_R(X^R)(s)|X^R(s)|^{2\sigma}X^R(s) \; ds \bigg| _{L^r_\omega L^{\infty}_{[0,T]}L_x^2} \\
&\leq |X_0|_{L^r_\omega L_x^2}+ c_{r,p} \;T^{\overline{\beta}}\; \big|\theta_R(X^R)|X^R|^{2\sigma+1} \; \big|_{L^{r}_{\omega} L^{r'}_{[0,T]} L^{p'}_x} \\
&=|X_0|_{L^r_\omega L_x^2}+ c_{r,p} \;T^{\overline{\beta}}\; \bigg| \bigg[ \int_0^T   \theta_R(X^R)^{r'}(t) \bigg( \int_{\Gamma} |X^R(t)|^{(2\sigma +1) p'} \; dx \bigg)^{\frac{r'}{p'}} \; dt \bigg]^{1/r'} \;  \bigg|_{L^{r}_{\omega}}\\
\intertext{Since $p'=\frac{2\sigma+2}{2\sigma+1}$, the right hand side equals}
&= |X_0|_{L^r_\omega L_x^2}+ c_{r,p} \;T^{\overline{\beta}}\; \bigg| \bigg( \int_0^T   \theta_R(X^R)^{r'}(t) |X^R(t)|^{(2\sigma +1) r'}_{L_x^{p}} \; dt \bigg)^{\frac{1}{r'}} \;  \bigg|_{L^{r}_{\omega}} \\
\intertext{Since $r \geq p=2\sigma+2$, we can apply {\it H\"{o}lder's inequality} to get}
&\leq |X_0|_{L^r_\omega L_x^2}+ c_{r,p} \;T^{\overline{\beta}}\; \bigg| \bigg( \int_0^T  [ \theta_R(X^R)(t) ]^{r/(2\sigma+1)} |X^R(t)|^{r}_{L_x^{p}} \; dt \bigg)^{\frac{2\sigma+1}{r}} T^{\frac{r-r'(2\sigma+1)}{rr'}} \;  \bigg|_{L^{r}_{\omega}} \\
\intertext{Letting $\widetilde{t}(\omega):= \inf \{s: \, |X^R(\omega)|_{L^r([0,s];L^p_x)} \geq 2R \}$, from definition of $\theta_R$ in \eqref{definition theta}, we have}
& = |X_0|_{L^r_\omega L_x^2}+ c_{r,p} \; T^{\tilde{\beta}} \; \bigg| \bigg( \int_0^{\widetilde{t}}  [ \theta_R(X^R)(t) ]^{r/(2\sigma+1)} |X^R(t)|^{r}_{L_x^{p}} \; dt \bigg)^{\frac{2\sigma+1}{r}} \;  \bigg|_{L^{r}_{\omega}} \\
\intertext{where $\tilde{\beta}=\overline{\beta}+{\frac{r-r'(2\sigma+1)}{rr'}} > 0$. Furthermore,}
& \leq |X_0|_{L^r_\omega L_x^2}+ c_{r,p} \; T^{\tilde{\beta}} \; \bigg| \bigg( \int_0^{\widetilde{t}} |X^R(t)|^{r}_{L_x^{p}} \; dt \bigg)^{\frac{2\sigma+1}{r}} \;  \bigg|_{L^{r}_{\omega}} \\
& \leq |X_0|_{L^r_\omega L_x^2}+ c_{r,p} \;T^{\tilde{\beta}}\; (2R)^{2\sigma} \; \big| X^R \big| _{L^{r}_{\omega} L^r_{[0,T]}L_x^{p}}.
\end{align*}

\noindent Similarly, using again the Strichartz type estimates,
\begin{align} \label{LrLrLpalready}
| \mathcal{T} X^R &|_{L^r_\omega L^r_{[0,T]}L_x^p} \nonumber \\ \nonumber
&\leq |S_{\beta}(t,0) X_0|_{L^r_\omega L^r_{[0,T]}L_x^p}+\bigg| \int_0^t S_{\beta}(t,s)\theta_R(X^R)(s)|X^R(s)|^{2\sigma}X^R(s) \; ds \bigg|_{L^r_\omega L^r_{[0,T]}L_x^p} \\ \nonumber
&\leq c_{r,p}\; T^{\beta/2}\; |X_0|_{L^r_\omega L_x^2}+c_{r,p} \; T^{\overline{\beta}}\;  \big|\theta_R(X^R)|X^R|^{2\sigma+1} \; \big|_{L^r_\omega L^{r'}_{[0,T]}L_x^{p'}} \\
&\leq c_{r,p}\; T^{\beta/2}\; |X_0|_{L^r_\omega L_x^2}+c_{r,p} \;T^{\tilde{\beta}}\; (2R)^{2\sigma} \; \big| X^R \big| _{L^r_\omega L^r_{[0,T]}L_x^p}. 
\end{align}

\noindent Hence, 
\begin{equation*}
\vertiii{\mathcal{T} X^R} \leq (1+ c_{r,p}\; T^{\beta/2})\; |X_0|_{L^r_\omega L_x^2}+2\; c_{r,p} \;T^{\tilde{\beta}}\; (2R)^{2\sigma} \; \big| X^R \big| _{L^r_\omega L^r_{[0,T]}L_x^p}. 
\end{equation*}

\noindent Letting $a:=2\;(1+ c_{r,p}\; T^{\beta/2})\; |X_0|_{L^r_\omega L_x^2}$, there exists $T \leq 1$ which depends only on $R, r, p$ s.t. $c_{r,p} \;T^{\tilde{\beta}}\; (2R)^{2\sigma} < \dfrac{1}{4}$, so that $\vertiii{\mathcal{T}X^R} \leq a$.

Now we prove that $\mathcal{T}$ is a contraction on $E(T,a)$. If $X$ and $Y$ $\in E(T,a)$, then by Strichartz estimates in Theorem \ref{stochastic strichartz},
\begin{align*}
\vertiii{\mathcal{T}X - \mathcal{T}Y} & =\bigg| \int_0^t S_{\beta}(t,s) \bigg[ \theta_R(X)(s)|X(s)|^{2\sigma}X(s) - \theta_R(Y)(s)|Y(s)|^{2\sigma}Y(s) \bigg] \; ds \bigg|_{L^r_\omega L^{\infty}_{[0,T]}L_x^2} \\
&+\bigg| \int_0^t S_{\beta}(t,s) \bigg[ \theta_R(X)(s)|X(s)|^{2\sigma}X(s) - \theta_R(Y)(s)|Y(s)|^{2\sigma}Y(s) \bigg] \; ds \bigg|_{L^r_\omega L^r_{[0,T]}L_x^p} \\
&\leq c_{r,p}\; T^{\overline{\beta}}\; \big| \theta_R(X)|X|^{2\sigma}X - \theta_R(Y)|Y|^{2\sigma}Y \big|_{L^r_\omega L^{r'}_{[0,T]}L_x^{p'}}. 
\end{align*}

\noindent By \cite[Proof of Theorem 4.1]{debouarddebussche} we have that
\begin{equation*}
 \big| \theta_R(X)|X|^{2\sigma}X - \theta_R(Y)|Y|^{2\sigma}Y \big|_{L^r_\omega L^{r'}_{[0,T]}L_x^{p'}} \leq C\;T^{\gamma}\; | X-Y |_{L^r_\omega L^r_{[0,T]}L_x^p}
\end{equation*}

\noindent with $\gamma=1-\frac{2\sigma+2}{r}$, hence 
\begin{equation*}
\vertiii{ \mathcal{T}X-\mathcal{T}Y} \leq c_{r,p}\; C\; T^{\tilde{\beta}+\gamma}\; \vertiii{X-Y},
\end{equation*}

\noindent which means that if $T \leq T_0$ for some sufficiently small $T_0$, then $\mathcal{T}$ is a strict contraction on $E(T,a)$. Since $T_0$ does not depend on the initial data, we can iterate the construction and obtain a global solution for the truncated equation.

\noindent{\it Proof of Step 2.} If $\rho \leq r$ and $(\rho,q)$ is an admissible pair, by Theorem \ref{stochastic strichartz} (i) and (ii.2) we get
\begin{equation*}
\begin{aligned}
|X^R|_{L^{\rho}_{\omega}L^{\rho}_{[0,T]}L_x^q} & \leq c_{\rho,q} T^{\beta/2} |X_0|_{L^{\rho}_{\omega}L^2_x} + \tilde{c}_{\rho,q}T^{\overline{\beta}} \big|\theta_R(X^R)|X^R|^{2\sigma+1} \; \big|_{L^{\rho}_{\omega} L^{r'}_{[0,T]} L^{p'}_x}\\
& \leq c_{\rho,q} T^{\beta/2} |X_0|_{L^{\rho}_{\omega}L^2_x} + \tilde{c}_{\rho,q}T^{\overline{\beta}} \big|\theta_R(X^R)|X^R|^{2\sigma+1} \; \big|_{L^r_\omega L^{r'}_{[0,T]}L_x^{p'}} \\
& \leq c_{\rho,q} T^{\beta/2} |X_0|_{L^{\rho}_{\omega}L^2_x} + \tilde{c}_{\rho,q}T^{\overline{\beta}} (2R)^{2\sigma} \; \big| X^R \big| _{L^r_\omega L^r_{[0,T]}L_x^p} < \infty.
\end{aligned}
\end{equation*}

\noindent{\it Proof of Step 3.} We claim that
\begin{equation} \label{commute group}
X^R(t)=S_{\beta}(t,0)X_0 + \mathrm{i} S_{\beta}(t,0) \int_0^t S_{\beta}(0,s) \theta_R(X^R)|X^R|^{2 \sigma} X^R(s) \, ds.
\end{equation}
Note that since $S_{\beta}(t,0) \in  \mathcal{L}(L^2_x,L^2_x)$, it is sufficient to check the Bochner integrability of the integrand. Indeed,
\begin{equation*}
\begin{aligned}
\int_0^T \big|S_{\beta}(0,s) \theta_R(X^R)|X^R|^{2 \sigma} X^R(s) \big|_{L^2_x} \, ds & = \int_0^T \big| \theta_R(X^R)|X^R|^{2 \sigma} X^R(s) \big|_{L^2_x} \, ds \\
& \leq \int_0^T \big| |X^R|^{2 \sigma+1}\big|_{L^2_x} \, ds = \int_0^T \Big[ \int_{\Gamma} |X^r|^{2(2 \sigma +1)} \, dx \Big]^{1/2} \, ds\\
& = |X^R|^{2\sigma+1}_{L^{2 \sigma+1}_{[0,T]} L^{2(2\sigma+1)}_x}
\end{aligned}
\end{equation*}

\noindent Since $\sigma < 2$, the exponent pair $(2\sigma+1, 2(2\sigma+1))$ is admissible and hence, by the additional integrability obtained at {\it Step 2}, 
\begin{equation*}
|X^R|_{L^{2 \sigma+1}_{\omega}L^{2 \sigma+1}_{[0,T]}L^{2(2\sigma+1)}_x} < \infty \text{ and hence }
|X^R|_{L^{2 \sigma+1}_{[0,T]} L^{2(2\sigma+1)}_x} < \infty \; \; \;  \mathbb{P}\mbox{-a.s.}
\end{equation*}

\noindent So,
\begin{equation} \label{bochner integrability}
\int_0^T \big|S_{\beta}(0,s) \theta_R(X^R)|X^R|^{2 \sigma} X^R(s) \big|_{L^2_x} \, ds \leq |X^R|^{2\sigma+1}_{L^{2 \sigma+1}_{[0,T]} L^{2(2\sigma+1)}_x} < \infty.
\end{equation}

\noindent Let now $Y(t):= S_{\beta}(0,t)X^R(t)$. Then, by \eqref{commute group}
\begin{equation*}
Y(t)=X_0 + \mathrm{i} \int_0^t S_{\beta}(0,s) \theta_R(X^R)|X^R|^{2 \sigma} X^R(s) \, ds \; \; \; \mathbb{P}\mbox{-a.s.} \, for \ all \ t \in[0,T].
\end{equation*}

\noindent Since the integrand belongs to $L^1([0,T];L^2_x)$, it follows that $(Y(t))_{t \in [0,T]}$ has a version with absolutely continuous trajectories in $[0,T]$ with values in $L^2_x$, which we again denote by $Y$. We show now that $(S_{\beta}(t,0)Y(t))_{t \in [0,T]}$, which is a version of $(X^R(t))_{t \in [0,T]}$, has continuous trajectories with values in $L^2_x$. We have that
\begin{equation*}
\begin{aligned}
\big|S_{\beta}(t,0)Y(t)&-S_{\beta}(s,0)Y(s)\big|_{L^2_x}\\ & \leq \big|S_{\beta}(t,0)Y(t) -S_{\beta}(s,0)Y(t)\big|_{L^2_x} + \big|S_{\beta}(s,0)Y(t)-S_{\beta}(s,0)Y(s)\big|_{L^2_x} \\
& = \big|[S_{\beta}(t,0)-S_{\beta}(s,0)] Y(t)\big|_{L^2_x} + \big|Y(t)-Y(s)\big|_{L^2_x}.
\end{aligned}
\end{equation*}

\noindent So, by the continuity of the group and the fact that $(Y(t))_{t \in [0,T]}$ is with absolutely continuous trajectories, we get the desired result.

\noindent{\it Proof of Step 4.} Note that since $S(t)$ is an isometry on $L^2_x$, the conservation of the $L^2_x$--norm of $X^R$ is equivalent with the $L^2_x$-norm conservation of $Y(t)$. Clearly, $\phi:[0,T] \to \mathbb{R}_+$ defined as
$
\phi(t):=|Y(t)|_{L^2_x}
$ is absolutely continuous, hence it is sufficient to show that $\phi'=0 \, \, \lambda- \mbox{a.e.}$ This holds true because
\begin{equation*}
\begin{aligned}
\frac{d}{dt} |Y(t)|_{L^{2}_x}^{2} &= 2 \Re \{\langle \frac{d}{dt}Y(t), Y(t) \rangle_{L^{2}_x}\}\\
& =2 \Re \{\mathrm{i} \langle S_{\beta}(0,t) \theta_R(X^R)|X^R|^{2\sigma}X^R(t),S_{\beta}(0,t) X^R(t) \rangle_{L^2_x}\} \\
& =2\Re \{\mathrm{i} \langle \theta_R(X^R)|X^R|^{2\sigma}X^R(t), X^R(t) \rangle_{L^2_x}\} = 0.
\end{aligned}
\end{equation*}
\noindent Hence, $
\phi(t)=\phi(0)$ for all $ t \in [0,T],
$
which completes our proof.
\end{proof}

\subsection{Proofs of results from Subsection \ref{convergence of approximate solutions}}

\begin{proof}[\bf{Proof of Theorem \ref{well-posedness}}]

\noindent \textbf{Well-posedness in} $\bm{C([0,T],L^2(\Gamma))}$. Consider the complete metric function space
\begin{equation*}
E(T,a)= \{u \in C([0,T],L^2(\Gamma)): \| u  \|_{C([0,T],L^2(\Gamma))} := \sup_{t \in [0,T]} \| u(t) \|_{L^2(\Gamma)} \leq a\}, 
\end{equation*}

\noindent where $T, a \in \mathbb{R}_+$ will be chosen later. Denote by $\uptau(u_n)(t)$ the right hand side of \eqref{mildform}.

First of all, note that if $u_n \in C([0,T]; L^2(\Gamma))$ then $F(u_n) \in L^{\infty}([0,T];L^2(\Gamma))$. Hence, we can use the same argument as in Step 3 of the proof of Theorem \ref{solution truncated} to deduce that $\uptau(u_n) \in C([0,T];L^2(\Gamma))$. Moreover,
\begin{equation*}
\| \uptau(u_n) \|_{C([0,T],L^2(\Gamma))} \leq \| u_0 \|_{L^2(\Gamma)} +  (2R^{\sigma}) \, T \,\|  u_n \|_{C([0,T],L^2(\Gamma))}.
\end{equation*}

\noindent Hence, taking for instance $a=2 \| u_0 \|_{L^2(\Gamma)}$ and $T \leq \dfrac{1}{2 (2 R)^{\sigma}}$, $\uptau:E(T,a) \to E(T,a)$ is well-defined.

Let now $u,v \in E(T,a)$, with $T$ and $a$ as before. Then,
\begin{equation*}
\begin{aligned}
\| \uptau(u)(t)-\uptau(v)(t) \|_{L^2(\Gamma)} & \leq \int_0^t \| F(u)-F(v) \|_{L^2(\Gamma)}\, ds \leq C(R,\sigma) \int_0^t \| u(s) -v(s) \|_{L^2(\Gamma}) \, ds,
\end{aligned}
\end{equation*}
\noindent Therefore,
\begin{equation*}
\| \uptau(u) -\uptau(v) \|_{C([0,T],L^2(\Gamma))} \leq C(R,\sigma) \, T \, \| u-v \|_{C([0,T],L^2(\Gamma))}.
\end{equation*}

\noindent Taking now $T \leq \min\bigg( \dfrac{1}{2 (2R)^{\sigma}}, \dfrac{1}{C(R,\sigma)} \bigg)$, we deduce that $\uptau :E(T,a) \to E(T,a)$ is a strict contraction, hence, there exists a unique solution $u_n \in C([0,T],L^2(\Gamma))$ to \eqref{mildform}. Since $T$ does depend only on the nonlinearity $F$, we can reiterate to get a global solution.

\noindent \textbf{Well-posedness in} $\bm{C([0,T],D(\mathcal{E}))}$. By the properties of $F$,
\begin{equation} \label{nonlinearity estimate}
\| F(u_n) \|_{D(\mathcal{E})} \leq C(R,\sigma) \| u_n \|_{D(\mathcal{E})}.
\end{equation}

\noindent Since $u_0 \in D(\mathcal{E}) \subset L^2(\Gamma)$, by the first part we have a unique solution $u_n \in C([0,T],L^2(\Gamma))$.

\noindent We show that $u_n$ belongs to $C([0,T],D(\mathcal{E}))$.
To this end, note that $u_n$ is the limit in $C([0,T];L^2(\Gamma))$ of the sequence $(u_n^k)_k$ constructed as follows:
\begin{equation*}
\begin{aligned}
&u_n^0=u_0\\
& u_n^{k+1}(\cdot)=\uptau(u_n^k)(\cdot)=S_{n}(\cdot,0)) u^0 + \mathrm{i} \int_{0}^{\cdot} S_{n} (\cdot,s) F(u_n^k)(s) \, ds, \quad \text{for all } k \geq 0. 
\end{aligned}
\end{equation*}

\noindent So, using \eqref{nonlinearity estimate} we get
\begin{equation*}
\begin{aligned}
\| u_n^{k+1} (t) \|_{D(\mathcal{E})} &\leq \| u^0 \|_{D(\mathcal{E})} + \int_0^t \| F(u_n^k(s)) \|_{D(\mathcal{E})} \, ds \leq \| u_n^0 \|_{D(\mathcal{E})} + C(R,\sigma) \int_0^t \| u_n^k(s) \|_{D(\mathcal{E})} \, ds.
\end{aligned}
\end{equation*}

\noindent Hence, if $v^k(t):= \| u_n^k(t) \|_{D(\mathcal{E})}$, we have for all $t \geq 0$ and $k \geq 0$
\begin{equation*}
v^{k+1}(t) \leq v^0 + C \int_0^t v^k(s) \, ds.
\end{equation*}

\noindent Hence, by \cite[Lemma 3.2 (i)]{debouarddebussche},
\begin{equation*}
\| u_n^k(t) \|_{D(\mathcal{E})} \leq \| u_0 \|_{D(\mathcal{E})} \mathrm{e}^{C(R,\sigma) T} , \quad \forall t \leq T.
\end{equation*}

\noindent Thus, for each fixed $t$, $(u_n^k(t))_k$ is a bounded sequence in the Hilbert space $D(\mathcal{E})$, hence there exists a subsequence $(u_n^{k_l})_l$ which is weakly convergent to some limit in $D(\mathcal{E})$. Since
\begin{equation*}
u_n^k(t) \underset{k \to \infty}\longrightarrow u_n(t) \in L^2(\Gamma),
\end{equation*}

\noindent we get that $u_n(t) \in D(\mathcal{E})$ for all $t \in [0,T]$. Since the norm on $D(\mathcal{E})$ is lower semi-continuous w.r.t. the weak topology, we also get
\begin{equation} \label{exponential estimate}
\| u_n(t) \|_{D(\mathcal{E})} \leq \|u_0\|_{D(\mathcal{E})} \mathrm{e}^{C(R,\sigma) T}, \quad \forall t \leq T.
\end{equation}

\noindent To show that $u_n \in C([0,T],D(\mathcal{E}))$, note first that $(S(t))_{t \in \mathbb{R}}$ is a continuous group on $D(\mathcal{E})$, since for all $u_0 \in D(\mathcal{E})$ and $t_k \to 0$, by the spectral representation and the spectral measure property \eqref{spectralmeasureproperty},
\begin{equation*}
\|S(t_k)u_0 - S(t)u_0  \|^2_{D(\mathcal{E})} = \int_{\sigma(-\Delta_{\Gamma})} (\lambda+M) |1 - cos[(t_k-t)\lambda]| \, d \mu_{u_0,u_0} (\lambda)
\end{equation*}

\noindent converges to $0$ when $k \to \infty$ by dominated convergence, since $\lambda + M \in L^1(\sigma(-\Delta_{\Gamma}), d\mu_{u_0,u_0})$ because $u_0 \in D(\mathcal{E})$.

\noindent Then, the desired continuity follows by dominated convergence theorem and \eqref{nonlinearity estimate}, since  
\begin{align*}
\| S(t_n-s) F(u(s))\|_{D(\mathcal{E})} &\leq \| F(u(s)) \|_{D(\mathcal{E})} \leq C(R,\sigma) \|u_0\|_{D(\mathcal{E})} \mathrm{e}^{C(R,\sigma) T} \\
& \leq C(R,\sigma) \| u_0\|_{D(\mathcal{E})} \mathrm{e}^{C(R,\sigma) T} \in L^1([0,T]). \qedhere
\end{align*} 
\end{proof}

\begin{proof}[\bf Proof of Theorem \ref{convergence}] Denote by $\beta_{\varepsilon}(\cdot):= \int_0^{\cdot} \frac{1}{\varepsilon} m\big(\frac{s}{\varepsilon^2}\big) \, ds$, with the process $m$ as in \ref{H0}. Since by hypothesis, $\beta_{\varepsilon}$ converges to $\beta$ converges in distribution on $C([0,T];\mathbb{R})$ for all $T>0$ it is sufficient to show that the mapping
\begin{equation*}
C([0,T];\mathbb{R}) \ni n \longmapsto u_n \in C([0,T];D(\mathcal{E})) 
\end{equation*}

\noindent is continuous, where $u_n$ is the solution of \eqref{mildform}. First of all, since $\| \cdot \|_{D(\mathcal{E})}$ is equivalent with $\| \cdot \|_{H^1}$ \eqref{equivnorms}, one can easily check that
\begin{equation*}
\|F(v) - F(w) \|_{D(\mathcal{E})} \leq C(R,\sigma)(1+ \| w \|_{D(\mathcal{E})}) \| v- w \|_{D(\mathcal{E})},
\end{equation*}

\noindent for all $v$ and $w$ in $D(\mathcal{E})$. Consequently, if $n_k \underset{k \to \infty}\longrightarrow n \in C([0,T];\mathbb{R})$, using also the fact that $S(t)$ is an isometry on $D(\mathcal{E})$,
\begin{equation*}
\begin{aligned}
\| u_{n_k}(t) - u_n(t) \|_{D(\mathcal{E})} & \leq \int_0^t \Big\| [S_{n_k} (t,s) - S_{n} (t,s) ] F(u_n)(s) \Big\|_{D(\mathcal{E})} \, ds \\
& \quad + \int_0^t \Big\| S_{n_k} (t,s) [ F(u_{n_k})(s) - F(u_n)(s) ]  \Big\|_{D(\mathcal{E})} \, ds \\
& \leq \int_0^T \Big\| [S_{n_k} (\cdot,s) - S_{n} (\cdot,s) ] F(u_n)(s) \Big\|_{C([0,T];D(\mathcal{E}))} \, ds \\
& \quad + C(R,\sigma) [1+ \| u_n \|_{C([0,T];D(\mathcal{E}))}] \int_0^t \| u_{n_k}(s)- u_n(s) \|_{D(\mathcal{E})} \, ds.
\end{aligned}
\end{equation*}

\noindent If we denote the first term in r.h.s. of the last inequality by $a_k$, we get by Gr\"{o}nwall lemma
\begin{equation*}
\| u_{n_k} - u_n \|_{C([0,T];D(\mathcal{E}))} \leq a_k \mathrm{e}^{C(R,\sigma) [1+ \| u_n \|_{C([0,T];D(\mathcal{E}))}]T}.
\end{equation*}

\noindent But $a_k \mathop{\longrightarrow}\limits_{k \to \infty} 0 $ by dominated convergence, since
\begin{equation*}
\begin{aligned}
\Big\| [S_{n_k} (\cdot,s) - S_{n} (\cdot,s) ] F(u_n)(s) \Big\|_{C([0,T];D(\mathcal{E}))} \leq 2 \| F(u_n)(s) \|_{D(\mathcal{E})} \leq 2 |f|_{\infty} \| u_n \|_{C([0,T];D(\mathcal{E}))}
\end{aligned}
\end{equation*}
\noindent and 
\begin{equation*}
\begin{aligned}
\Big\| [S_{n_k} (\cdot,s) & - S_{n} (\cdot,s) ] F(u_n)(s) \Big\|^2_{C([0,T];D(\mathcal{E}))} \\
& \leq 2 \int_{\sigma(-\Delta_{\Gamma})} (\lambda + M) [2 \wedge |\lambda| | n_k(\cdot) - n_k(s) - (n(\cdot) - n(s)) | ] \, d \mu_{F(u_n),F(u_n)} (\lambda)
\end{aligned}
\end{equation*}

\noindent which converges to $0$ for all $s \in [0,T]$, again by dominated convergence since $F(u_n) \in D(\mathcal{E})$.
\end{proof}

\subsection{Proofs of results from Subsection \ref{well posedness star-graph}}

\begin{proof}[\bf Proof of Proposition \ref{derivativesemigroup}]

We need the following claim whose proof is postponed right after the proof of Proposition \ref{derivativesemigroup}.

\begin{claim} \label{claim1} If $u_0 \in D(\mathcal{E}_{\star})$, then
$\lim_{\varepsilon \to 0}\mathrm{e}^{-(\mathrm{i}t + \varepsilon)H_{\star}}u_0 =\mathrm{e}^{- \mathrm{i}t H_{\star}} u_0  \ \text{ in } H^1(\Gamma).$
\end{claim}

\noindent As in the proof of Proposition \ref{prop 2.13}, let $n_+(AB^{\dag})$ be the number of (negative) eigenvalues $(\lambda_l)_l$ of $H_{\star}$ and denote by $(\varphi_l)_l$ the corresponding eigenfunctions. By \cite[Proposition 3.2]{grecuignatjphysa} and e.g. \cite[Theorem VII.11]{reedsimon}, the continuous spectrum of $H_{\star}$ is $[0,\infty)$. Consequently, we have the following representation:
\begin{equation*}
\mathrm{e}^{-(\mathrm{i} t +\varepsilon)H_{\star}} u_0(x)= \frac{1}{2 \pi \mathrm{i}} \lim_{\delta \to 0} \int_0^{\infty} \mathrm{e}^{-(\mathrm{i}t+\varepsilon)\lambda} [R_{\lambda+ \mathrm{i} \delta} - R_{\lambda-\mathrm{i}\delta}] u_0 (x) \, d\lambda + \sum_{l=1}^{n_+(AB^{\dag})} \langle u_0,\varphi_l \rangle \mathrm{e}^{-(\mathrm{i}t + \varepsilon) \lambda_l} \varphi_l
\end{equation*}

\noindent and the limit is in $L^2(\Gamma_{\star})$, where $R_z := (H_{\star}-z)^{-1}$. Taking into account the explicit kernel obtained in \cite[Lemma 4.2]{laplacians}, we arrive at:
\begin{equation*}
\begin{aligned}
&\mathrm{e}^{-(\mathrm{i}t + \varepsilon)H_{\star}} u_0 (x_i)\\ & = \dfrac{1}{2 \pi \mathrm{i}} \lim_{\delta \to 0} \int_0^{\infty} \mathrm{e}^{-(\mathrm{i}t + \varepsilon) \lambda} \Big\{\sum_{j=1}^n \int_{I_j} [ r(\sqrt{\lambda + \mathrm{i} \delta}; x_i,y_j) - r(\sqrt{\lambda - \mathrm{i} \delta};x_i,y_j) ] u_0(y_j) \, dy_j \Big\}  \, d\lambda\\
& + \sum_{l=1}^{n_+(AB^{\dag})} \langle u_0,\varphi_l \rangle \mathrm{e}^{-(\mathrm{i}t + \varepsilon) \lambda_l} \varphi_l(x),
\end{aligned}
\end{equation*}

\noindent with the matrix kernel $r(z;x;y)$ s.t. $r(z;x_i,y_j)= \dfrac{\mathrm{i}}{2 z} \mathrm{e}^{\mathrm{i}z |x_i-y_j|} \delta_{ij} + \dfrac{\mathrm{i}}{2z} \mathrm{e}^{\mathrm{i}z x_i} [G(z;A,B)]_{i,j} \mathrm{e}^{\mathrm{i} z y_j}$,

\noindent where $z^2 \in \mathbb{C} \setminus \sigma(H_{\star})$, $x_i$ and $y_j$ belong to the edges $I_i$ and $I_j$ respectively, $i,j=\overline{1,n}$ and the matrix $G$ is given by $G(z;A,B)=-(A+\mathrm{i}zB)^{-1}(A-\mathrm{i}zB)$.

\begin{claim} \label{claim2} If $u_0 \in D(\mathcal{E}_{\star})$, then for all $ 1 \leq i \leq n$,
\begin{equation*}
\begin{aligned}
&\partial_{x_i} (\mathrm{e}^{-(\mathrm{i}t + \varepsilon)H_{\star}} u_0(x_i)) \\
& =\dfrac{1}{2 \pi \mathrm{i}} \lim_{\delta \to 0} \partial_{x_i} \int_0^{\infty} \mathrm{e^{-(\mathrm{i}t + \varepsilon) \lambda}} \Big\{ \sum_{j=1}^n \int_{I_j} [r(\sqrt{\lambda + \mathrm{i} \delta}; x_i,y_j) - r(\sqrt{\lambda - \mathrm{i} \delta};x_i,y_j) ] u_0(y_j) \, dy_j  \Big\} \, d\lambda \\
&+ \sum_{l=1}^{n_+(AB^{\dag}} \langle u_0,\varphi_l \rangle \mathrm{e}^{-(\mathrm{i}t + \varepsilon) \lambda_l} \varphi'_l(x_i).
\end{aligned}
\end{equation*}
\end{claim}

Let us now deal with the first term in the above r.h.s. and notice that
\begin{align} \label{integrandI}
\partial_{x_i} \nonumber & \int_0^{\infty} \mathrm{e^{-(\mathrm{i}t + \varepsilon) \lambda}} \Big\{ \sum_{j=1}^n  \int_{I_j} [r(\sqrt{\lambda + \mathrm{i} \lambda}; x_i,y_j) - r(\sqrt{\lambda - \mathrm{i} \delta};x_i,y_j) ] u_0(y_j) \, dy_j  \Big\} \, d\lambda  \\ \nonumber
&= \int_0^{\infty} \mathrm{e^{-(\mathrm{i}t + \varepsilon) \lambda}} \Big\{ \sum_{j=1}^n  \int_{I_j} \partial_{x_i} [r(\sqrt{\lambda + \mathrm{i} \delta}; x_i,y_j) - r(\sqrt{\lambda - \mathrm{i} \delta};x_i,y_j) ] u_0(y_j) \, dy_j  \Big\} \, d\lambda \\ 
&=: \int_0^{\infty} \mathrm{e^{-(\mathrm{i}t + \varepsilon) \lambda}} \sum_{j=1}^n [I_{(j)}(\sqrt{\lambda + \mathrm{i} \delta}; x_i) - I_{(j)}(\sqrt{\lambda - \mathrm{i} \delta};x_i) ] \, d\lambda. 
\end{align}

\noindent We treat now the integrands $I_{(j)}(\sqrt{\lambda \pm \mathrm{i} \delta};x_i)$:
\begin{align*}
&I_{(j)} (\sqrt{\lambda \pm \mathrm{i} \delta}; x_i) =\int_{I_j} \partial_x r(\sqrt{\lambda \pm \mathrm{i} \delta};x_i,y_j) u_0(y_j) \, dy_j\\
&= \int_{I_j} \partial_{x_i} \Bigg[ \frac{\mathrm{i}}{2 \sqrt{\lambda \pm \mathrm{i} \delta}} \mathrm{e}^{\mathrm{i} \sqrt{\lambda \pm \mathrm{i} \delta} |x_i-y_j|} \delta_{ij} + \dfrac{\mathrm{i}}{2 \sqrt{\lambda \pm \mathrm{i} \delta}} \mathrm{e}^{\mathrm{i} \sqrt{\lambda \pm \mathrm{i} \delta} x_i} [G(\sqrt{\lambda \pm \mathrm{i} \delta}; A,B)]_{i,j} \mathrm{e}^{\mathrm{i} \sqrt{\lambda \pm \mathrm{i} \delta} y_j} \Bigg] u_0(y_j) \, dy_j \\
&=\int_{I_j} \Bigg[ \frac{\mathrm{i}}{2 \sqrt{\lambda \pm \mathrm{i} \delta}} \mathrm{i} \sqrt{\lambda \pm \mathrm{i} \delta} \sign(x_i-y_j) \mathrm{e}^{\mathrm{i} \sqrt{\lambda \pm \mathrm{i} \delta} |x_i-y_j|} \delta_{ij}\\
& \qquad +\dfrac{\mathrm{i}}{2 \sqrt{\lambda \pm \mathrm{i} \delta}}  \mathrm{i} \sqrt{\lambda \pm \mathrm{i} \delta} \mathrm{e}^{\mathrm{i} \sqrt{\lambda \pm \mathrm{i} \delta} x_i} [G(\sqrt{\lambda \pm \mathrm{i} \delta}; A,B)]_{i,j} \mathrm{e}^{\mathrm{i} \sqrt{\lambda \pm \mathrm{i} \delta} y_j} \Bigg] u_0(y_j) \, dy_j \\
&=  \int_{I_j} \partial_{y_j} \Bigg[ \frac{-\mathrm{i}}{2 \sqrt{\lambda \pm \mathrm{i} \delta}} \mathrm{e}^{\mathrm{i} \sqrt{\lambda \pm \mathrm{i} \delta} |x_i-y_j|} \delta_{ij} + \dfrac{\mathrm{i}}{2 \sqrt{\lambda \pm \mathrm{i} \delta}} \mathrm{e}^{\mathrm{i} \sqrt{\lambda \pm \mathrm{i} \delta} x_i} [G(\sqrt{\lambda \pm \mathrm{i} \delta}; A,B)]_{i,j} \mathrm{e}^{\mathrm{i} \sqrt{\lambda \pm \mathrm{i} \delta} y_j} \Bigg] u_0(y_j) \, dy_j.
\end{align*}

\noindent Integrating by parts, we continue with
\begin{equation*}
\begin{aligned}
&= \dfrac{\mathrm{i}}{2 \sqrt{\lambda \pm \mathrm{i} \delta}} \mathrm{e}^{\mathrm{i} \sqrt{\lambda \pm \mathrm{i} \delta} x_i} \Big[ \delta_{ij} - [G(\sqrt{\lambda \pm \mathrm{i} \delta}; A, B)]_{i,j}   \Big] u_{0,j}(0)\\
&- \int_{I_j} \Bigg[ \frac{-\mathrm{i}}{2 \sqrt{\lambda \pm \mathrm{i} \delta}} \mathrm{e}^{\mathrm{i} \sqrt{\lambda \pm \mathrm{i} \delta} |x_i-y_j|} \delta_{ij} + \dfrac{\mathrm{i}}{2 \sqrt{\lambda \pm \mathrm{i} \delta}} \mathrm{e}^{\mathrm{i} \sqrt{\lambda \pm \mathrm{i} \delta} x_i} [G(\sqrt{\lambda \pm \mathrm{i} \delta}; A,B)]_{i,j} \mathrm{e}^{\mathrm{i} \sqrt{\lambda \pm \mathrm{i} \delta} y_j} \Bigg] u'_{0,j}(y_j) \, dy_j
\end{aligned}
\end{equation*}

\noindent Hence, 
\begin{align*}
I_{(j)}(\sqrt{\lambda + \mathrm{i} \delta}; & x_i)- I_{(j)} (\sqrt{\lambda - \mathrm{i} \delta}; x_i) = \dfrac{\mathrm{i}}{2 \sqrt{\lambda + \mathrm{i} \delta}} \mathrm{e}^{\mathrm{i} \sqrt{\lambda + \mathrm{i} \delta} x_i} \Big[ \delta_{ij} - [G(\sqrt{\lambda + \mathrm{i} \delta}; A, B)]_{i,j}   \Big] u_{0,j} (0)\\
& \quad \qquad \qquad \qquad \qquad - \dfrac{\mathrm{i}}{2 \sqrt{\lambda - \mathrm{i} \delta}} \mathrm{e}^{\mathrm{i} \sqrt{\lambda - \mathrm{i} \delta} x_i} \Big[ \delta_{ij} - [G(\sqrt{\lambda - \mathrm{i} \delta}; A, B)]_{i,j}  \Big] u_{0,j} (0) \\
& - \int_{I_j} [ r(\sqrt{\lambda + \mathrm{i} \delta}; x_i,y_j) - r(\sqrt{\lambda - \mathrm{i} \delta};x_i,y_j)  ] u'_{0,j}(y_j) \, dy_j \\
&  +2 \int_{I_j} \frac{\mathrm{i}}{2 \sqrt{\lambda + \mathrm{i} \delta}} \mathrm{e}^{\mathrm{i} \sqrt{\lambda + \mathrm{i} \delta} |x_i-y_j|} \delta_{ij} u'_{0,j}(y_j) - \dfrac{\mathrm{i}}{2 \sqrt{\lambda - \mathrm{i} \delta}} \mathrm{e}^{\mathrm{i} \sqrt{\lambda - \mathrm{i} \delta} |x_i-y_j|} \delta_{ij} u'_{0,j}(y_j) \, dy_j.
\end{align*}

\noindent Substituting this back in \eqref{integrandI}, we get that
\begin{align} \label{threelimits}
& \partial_{x_i} ( \mathrm{e}^{- (\mathrm{i} t + \varepsilon )H_{\star}} u_0(x_i) ) = \dfrac{1}{2 \pi \mathrm{i}} \lim_{\delta \to 0} \int_0^{\infty} \mathrm{e}^{-(\mathrm{i} t + \varepsilon) \lambda} \sum_{j=1}^n \Big\{  \dfrac{\mathrm{i}}{2 \sqrt{\lambda + \mathrm{i} \delta}} \mathrm{e}^{\mathrm{i} \sqrt{\lambda + \mathrm{i} \delta} x_i} [ \delta_{ij} - [G(\sqrt{\lambda + \mathrm{i} \delta}; A, B)]_{i,j} ] \nonumber \\ \nonumber
& \qquad \qquad \qquad \qquad \qquad - \dfrac{\mathrm{i}}{2 \sqrt{\lambda - \mathrm{i} \delta}} \mathrm{e}^{\mathrm{i} \sqrt{\lambda - \mathrm{i} \delta} x_i} [ \delta_{ij} - [G(\sqrt{\lambda - \mathrm{i} \delta};A,B)]_{i,j} ] \Big\} u_{0,j}(0) \, d \lambda \\ \nonumber
& - \dfrac{1}{2 \pi \mathrm{i}} \lim_{\delta \to 0} \int_0^{\infty} \mathrm{e}^{-(\mathrm{i} t + \varepsilon) \lambda} \sum_{j=1}^n \int_{I_j} [ r(\sqrt{\lambda + \mathrm{i} \delta}; x_i,y_j) - r(\sqrt{\lambda - \mathrm{i} \delta};x_i,y_j)  ] u'_{0,j}(y_j) \, dy_j  \, d\lambda\\ \nonumber
& + \dfrac{1}{2 \pi \mathrm{i}} \lim_{\delta \to 0} \int_0^{\infty} 2 \mathrm{e}^{-(\mathrm{i} t + \varepsilon ) \lambda}  \int_{I_i} \frac{\mathrm{i}}{2 \sqrt{\lambda + \mathrm{i} \delta}} \mathrm{e}^{\mathrm{i} \sqrt{\lambda + \mathrm{i} \delta} |x_i-y_i|} u'_{0,i}(y_i) - \dfrac{\mathrm{i}}{2 \sqrt{\lambda - \mathrm{i} \delta}} \mathrm{e}^{\mathrm{i} \sqrt{\lambda - \mathrm{i} \delta} |x_i-y_i|} u'_{0,i}(y_i) \, dy_i \, d\lambda\\
& =:J_1+J_2+J_3
\end{align}

\noindent By the spectral theorem, $J_2 = \mathrm{e}^{-(\mathrm{i} t + \varepsilon) H_{\star}} P_c u'_0(x_i)$. By \eqref{extended},
\begin{equation*}
\begin{aligned}
J_3= \dfrac{1}{2 \pi \mathrm{i}} \lim_{\delta \to 0} \int_0^{\infty} 2 \mathrm{e}^{-(\mathrm{i} t + \varepsilon ) \lambda} \int_{\mathbb{R}} \frac{\mathrm{i}}{2 \sqrt{\lambda + \mathrm{i} \delta}} \mathrm{e}^{\mathrm{i} \sqrt{\lambda + \mathrm{i} \delta} |x_i-y_i|} \widetilde{u'}_{0,i}(y_i) - \dfrac{\mathrm{i}}{2 \sqrt{\lambda - \mathrm{i} \delta}} \mathrm{e}^{\mathrm{i} \sqrt{\lambda - \mathrm{i} \delta} |x_i-y_i|} \widetilde{u'}_{0,i}(y_i) \, dy_i \, d\lambda.
\end{aligned}
\end{equation*}

\noindent Making the change of variable $k=\sqrt{\lambda}$ and $k=-\sqrt{\lambda}$ in the first and second integral, respectively, and taking into account that $\lim_{\delta \to 0} \sqrt{k^2 \pm \mathrm{i} \delta} = \pm |k|$, by the dominated convergence theorem we get
\begin{equation*}
J_3= \dfrac{1}{\pi} \int_{\mathbb{R}} \mathrm{e}^{-(\mathrm{i} t + \varepsilon) k^2} \int_ \mathbb{R} \mathrm{e}^{\mathrm{i} k | x_i -y_i|} \widetilde{u'}_{0,i}(y_i) \, dy_i \, dk. \qquad \qquad \qquad \qquad \qquad \qquad \qquad \qquad \quad
\end{equation*}

\noindent Note that letting $\varepsilon$ go to zero, we recover precisely $2 \mathrm{e}^{ \mathrm{i} t \Delta_{\mathbb{R}}} \widetilde{u'}_0$ \cite[(3.6)-(3.7)]{angulopava}.

Proceeding with the same changes of variables in the case of the first term in \eqref{threelimits}, from the properties of the matrix $G$ in \cite[Proof of Lemma 3.3]{grecuignatjphysa}, again by the dominated convergence theorem we have
\begin{align} \label{finalextraterm}
J_1 &= \dfrac{1}{2 \pi} \int_{\mathbb{R}} \mathrm{e}^{-(\mathrm{i} t + \varepsilon) k^2} \mathrm{e}^{\mathrm{i}k x_i} \sum_{j=1}^n [\delta_{ij} - [ G(k;A,B)]_{i,j}] u_{0,j}(0) \, dk \\ \nonumber
& = \dfrac{1}{2 \pi} \int_{\mathbb{R}} \mathrm{e}^{-(\mathrm{i} t + \varepsilon) k^2} \mathrm{e}^{\mathrm{i}k x_i} \{ [ \mathbb{I}_n -G(k;A,B)] u_{0}(0)\}_i \, dk.
\end{align}

\noindent Thus, letting $\varepsilon$ tend to $0$, we obtain
\begin{equation} \label{derivativesemigroupwhole}
\begin{aligned}
\partial_{x_i}(\mathrm{e}^{- \mathrm{i} t H_{\star}} u_0(x_i)) & = \mathrm{e}^{-\mathrm{i} t H_{\star}} P_c u'_0(x_i) +  2 \mathrm{e}^{-\mathrm{i} t \Delta_{\mathbb{R}}} \widetilde{u'_0}(x_i) + \sum_{l=1}^{n_+(AB^{\dag})} \langle u_0,\varphi_l \rangle \mathrm{e}^{-(\mathrm{i}t + \varepsilon) \lambda_l} \varphi_l'(x_i)\\
& \quad + \lim_{\varepsilon \to 0}  \dfrac{1}{2 \pi} \int_{\mathbb{R}} \mathrm{e}^{-(\mathrm{i} t + \varepsilon) k^2} \mathrm{e}^{\mathrm{i}k x_i} \{ [ \mathbb{I}_n -G(k;A,B)] u_{0}(0)\}_i \, dk.
\end{aligned}
\end{equation}

\noindent By \cite[Lemma 2.1.3]{berkolaiko} we can write
\begin{equation} \label{otherGexpression}
G(k;A,B) = -P_D + P_N -(\Lambda + \mathrm{i} k)^{-1}(\Lambda - \mathrm{i} k)P_R,
\end{equation}

\noindent with $P_D,P_N,P_R$ and $\Lambda$ as in Theorem \ref{self-adjointconditions}. Now we use that by hypothesis there is no Robin part, i.e.$P_R=0$, so plugging \eqref{otherGexpression} into \eqref{finalextraterm}, we get that
\begin{equation*}
J_1 = \dfrac{1}{2 \pi} \int_{\mathbb{R}} \mathrm{e}^{-(\mathrm{i} t + \varepsilon) k^2} \mathrm{e}^{\mathrm{i}k x_i} [ P_D u_{0}(0)]_i \, dk =0,
\end{equation*}

\noindent since $u_0 \in D(\mathcal{E}_{\star})$, hence $P_D u_0(0)=0$.
\end{proof}

\begin{proof}[\bf Proof of Claim \ref{claim1}] From the equivalence of $H^1$ and $D(\mathcal{E_{\star}})-$norms by Proposition \ref{equivnorms}, the proof resumes to show that $ \lim_{\varepsilon \to 0} \, \| \mathrm{e}^{-(\mathrm{i} t + \varepsilon)H_{\star}}u_0 - \mathrm{e}^{-\mathrm{i}t H_{\star}} u_0\|_{D(\mathcal{E}_{\star})}=0$. Indeed, letting $h_{\varepsilon}(\lambda):=\mathrm{e}^{-(\mathrm{i} t + \varepsilon)\lambda} - \mathrm{e}^{-\mathrm{i}t \lambda}$, by the spectral representation of the form we have
\begin{equation*}
\| h_{\varepsilon}(H_{\star})u_0 \|_{D(\mathcal{E}_{\star})} = \int_{\sigma(H_{\star})} (\lambda +M)\,d \mu_{h_{\varepsilon}(H_{\star})u_0,h_{\varepsilon}(H_{\star})u_0}(\lambda).
\end{equation*}

\noindent By \cite[Theorem 3.1]{teschl}, for any self-adjoint operator $H$ and for every two bounded measurable functions $f$ and $g$ on $\mathbb{R}$, the spectral measure has the property
\begin{equation} \label{spectralmeasureproperty}
d \mu_{f(H)u,g(H)v} = f \bar{g} d \mu_{u,v}.
\end{equation}
\noindent Thus, 
\begin{equation*}
\| h_{\varepsilon}(H_{\star})u_0\|_{D(\mathcal{E}_{\star})} = \int_{\sigma(H_{\star})} (\lambda+M) |h_{\varepsilon}(\lambda)|^2 \,d \mu_{u_0,u_0}(\lambda) = \int_{\sigma(H_{\star})} (\lambda+M) |\mathrm{e}^{-\varepsilon \lambda}-1|^2 \,d \mu_{u_0,u_0}(\lambda) \longrightarrow 0,
\end{equation*}

\noindent by dominated convergence, since $\lambda \mapsto \lambda+M \in L^1(\sigma(H),d\mu_{u_0,u_0})$, because $u_0 \in D(\mathcal{E})$.
\end{proof}

\begin{proof}[\bf Proof of Claim \ref{claim2}] Let $\varepsilon>0$ and $u_0 \in D(\mathcal{E}_{\star})$. By the same equivalence of norms invoked in the proof of Claim \ref{claim1}, it sufficient to show that
\begin{equation} \label{spectralrepresentationDE}
\lim_{\delta \to 0} \dfrac{1}{2 \pi \mathrm{i}} \int_{\sigma(H_{\star})} \mathrm{e}^{-(\mathrm{i}t +\varepsilon)\lambda} [R_{\lambda + \mathrm{i} \delta} - R_{\lambda - \mathrm{i} \delta}] u_0 \, d\lambda = \mathrm{e}^{-(\mathrm{i}t + \varepsilon)H_{\star} }u_0 \quad \text{in} \ D(\mathcal{E}_{\star}).
\end{equation}

\noindent Since $\sigma(H_{\star}) \subset [-M,\infty)$, and setting for $z \in [-M,\infty)$,
\begin{equation*}
\begin{aligned}
& g^{\varepsilon}_{\delta}(z):=\dfrac{1}{2 \pi \mathrm{i}} \int_{-M}^{\infty} \mathrm{e}^{-(\mathrm{i} t + \varepsilon) \lambda} \Big[ \dfrac{1}{z-(\lambda + \mathrm{i} \delta)} - \dfrac{1}{ z- (\lambda- \mathrm{i}\delta)} \Big] \, d\lambda \\
& g^{\varepsilon}(z):=\mathrm{e}^{-(\mathrm{i} t + \varepsilon) z}.
\end{aligned}
\end{equation*}

\noindent \eqref{spectralrepresentationDE} rewrites as $\lim_{\delta \to 0} g^{\varepsilon}_{\delta}(H_{\star}) u_0 = g^{\varepsilon} (H_{\star}) u_0 \quad \text{in} \ D(\mathcal{E}_{\star})$.

\noindent But,by \eqref{spectralmeasureproperty}, $
\| g^{\varepsilon}_{\delta}(H_{\star}) u_0 - g^{\varepsilon}(H_{\star}) u_0 \|_{D(\mathcal{E}_{\star})}= \int_{-M}^{\infty} (\lambda+M) |g^{\varepsilon}_{\delta}(\lambda) - g^{\varepsilon}(\lambda)| \, d \mu_{u_0,u_0} (\lambda) $

\noindent which converges to $0$ by dominated convergence, since it is easy to check that $g^{\varepsilon}_{\delta} \xrightarrow[\delta \to 0]{} g^{\varepsilon}$ pointwise and boundedly on $[-M,\infty)$.
\end{proof}

\begin{proof}[\bf Proof of Theorem \ref{thm 2.16}] Let us first consider equation \eqref{truncated} for $R>0$, and let
\begin{equation*}
\widetilde{E}(T,a):=\{X^R \; (\mathcal{F}_t)\mbox{-adapted }: \; \vertiii{X} := |X^R|_{L^r_\omega(C([0,T];D(\mathcal{E}_{\star})))}+|X^R|_{L^r_{\mathcal{P}}(\Omega \times [0,T];W^{1,p}_x)} \leq a\}
\end{equation*}

\noindent Keeping the same notations as in the proof of Theorem \ref{solution truncated}, and using the fact that $S(t)$ is an isometry on $D(\mathcal{E}_{\star})$, but also that $\| \cdot \|_{D(\mathcal{E}_{\star})} \sim  \| \cdot \|_{H^1}$, we get
\begin{equation*}
\begin{aligned}
& |\mathcal{T}(X^R)|_{L^r_\omega(C([0,T];D(\mathcal{E}_{\star})))} \\
 & \leq |S_{\beta}(\cdot,0) X_0|_{L^r_\omega(C([0,T];D(\mathcal{E}_{\star})))} + \Big|\int_0^t S_{\beta}(t,s) \theta_R(X^R(s)) |X^R|^{2 \sigma}X^R(s) \, ds \Big|_{L^r_\omega(C([0,T];D(\mathcal{E}_{\star}))))} \\
& = |X_0|_{L^r_\omega(D(\mathcal{E}_{\star}))} +  c_1 \Big| \int_0^t S_{\beta}(t,s) \theta_R(X^R(s)) |X^R|^{2\sigma} X^R(s) \, ds  \Big|_{L^r_\omega(C([0,T];L^2_x))}\\
& \qquad \qquad \qquad \ +c_1 \Big| \int_0^t \dfrac{\partial}{\partial x}  S_{\beta}(t,s) \theta_R(X^R(s)) |X^R|^{2\sigma} X^R(s) \, ds  \Big|_{L^r_\omega(C([0,T];L^2_x))}\\
\end{aligned}
\end{equation*}

\noindent By Proposition \ref{derivativesemigroup},
\begin{align*}
|\mathcal{T}(X^R)&|_{L^r_\omega(C([0,T];D(\mathcal{E}_{\star}))} \\
& =|X_0|_{L^r_\omega(D(\mathcal{E}_{\star}))} + c_1 \Big| \int_0^t S_{\beta}(t,s) \theta_R(X^R(s)) |X^R|^{2\sigma} X^R(s) \, ds  \Big|_{L^r_\omega(C([0,T];L^2_x))}\\
& +c_1 (2 \sigma +1) \Big| \int_0^t S_{\beta}(t,s) \theta_R(X^R(s)) |X^R|^{2\sigma} \dfrac{\partial X^R}{\partial x} (s) \, ds  \Big|_{L^r_\omega(C([0,T];L^2_x))}\\
& + 2 c_1 (2 \sigma +1) \Big| \int_0^t S_{\mathbb{R}}(t-s) \theta_R(X^R(s)) |X^R|^{2\sigma} \dfrac{\partial \widetilde{X}^R}{\partial x} (s) \, ds  \Big|_{L^r_\omega(C([0,T];L^2_x))} \\
& + c_1 \Big| \int_0^t |\theta_R (X^R(s))| \sum_{l=1}^{n_+(AB^{\dag})} \big\{ \langle |X^R|^{2 \sigma} X^R(s), \varphi_l \rangle \mathrm{e}^{- \mathrm{i} [\beta(t) - \beta(s)] \lambda_l} \varphi'_l  \big\} \, ds  \Big|_{L^r_\omega(C([0,T];L^2_x))}
\end{align*}

\noindent where $\dfrac{\partial \widetilde{X}^R}{\partial x} (s)$ is understood in the sense of \eqref{extended}, $S_{\mathbb{R}}(t-s):=\mathrm{e}^{-\mathrm{i} [ \beta(t)-\beta(s) ] \Delta_{\mathbb{R}} }$ and $\{(\lambda_l, \varphi_l)\}_l$ are the eigencouples of $H_{\star}$. Applying the Strichartz-type estimates in Theorem \ref{stochastic strichartz}, which also hold for $S_{\mathbb{R}}$ \cite{debouarddebussche}, we get
\begin{align*}
|\mathcal{T}(X^R) & |_{L^r_\omega(C([0,T];D(\mathcal{E}_{\star}))} \\
& \leq |X_0|_{L^r_\omega(D(\mathcal{E}_{\star}))} + C(c_1,r,p) T^{\overline{\beta}} \Big|\theta_R(X^R) |X^R|^{2\sigma} X^R \Big|_{L^r_\omega L^{r'}_{[0,T]} L^{p'}_x}\\
& + 3 (2 \sigma+1)C(c_1,r,p)T^{\overline{\beta}} \Big|\theta_R(X^R) |X^R|^{2\sigma} \dfrac{\partial X^R}{\partial x} \Big|_{L^r_\omega L^{r'}_{[0,T]} L^{p'}_x}\\
& + \Big( \sum_{l=1}^{n_+(AB^{\dag})} \| \varphi_l \|_{L^p_x} \| \varphi_l \|_{L^2_x} \Big) T^{1/r'} \Big| \theta_R(X^R) |X^R|^{2 \sigma} X^R \Big|_{L^r_\omega L^{r'}_{[0,T]} L^{p'}_x}\\
& \leq |X_0|_{L^r_\omega(D(\mathcal{E}_{\star}))} + C\Big(c_1,r,p,\sum_{l=1}^{n_+(AB^{\dag})} \| \varphi_l \|_{L^p_x} \| \varphi_l \|_{L^2_x} \Big) T^{\hat{\beta}} (2R)^{2 \sigma} |X^R|_{L^r_\omega L^r_{[0,T]}L_x^p} \\
& + 3 (2 \sigma+1)C(c_1,r,p)T^{\overline{\beta}} \Big|\theta_R(X^R) |X^R|^{2\sigma} \dfrac{\partial X^R}{\partial x} \Big|_{L^r_\omega L^{r'}_{[0,T]} L^{p'}_x},
\end{align*}

\noindent for some $\hat{\beta} \geq 0$, where for the last inequality we used a similar estimate as \eqref{LrLrLpalready} from the proof of Theorem \ref{solution truncated}. 

\noindent Taking into account the definition of $\theta_R$ in \eqref{definition theta} and the fact that $p=2 \sigma+2$, in the case of the last term we get
\begin{equation*}
\begin{aligned}
 \Big|\theta_R(X^R) |X^R|^{2\sigma} \dfrac{\partial X^R}{\partial x} \Big|_{L^r_\omega L^{r'}_{[0,T]} L^{p'}_x}
 &=\Big|\theta_R(X^R)  \bigg( \int_{\Gamma_{\star}} |X^R|^{(p-2)p/(p-1)} \bigg| \dfrac{\partial X^R}{\partial x} \bigg|^{p/(p-1)} \, dx \bigg)^{(p-1)/p}\Big|_{L^r_\omega L^{r'}_{[0,T]}}.
\end{aligned}
\end{equation*} 

\noindent Applying H\"{o}lder inequality for the integral w.r.t. $x$, the above is less or equal
\begin{equation*}
\begin{aligned}
& \leq \bigg| \theta_R(X^R) \bigg( \bigg| |X^R|^{(p-2)p/(p-1)} \bigg|_{L^{\frac{p-1}{p-2}}_x} \times \bigg| \big| \dfrac{\partial X^R}{\partial x} \big|^{p/(p-1)} \bigg|_{L^{p-1}_x}  \bigg)^{(p-1/p)}  \bigg|_{L^r_\omega L^{r'}_{[0,T]}}\\
&=\bigg|  \theta_R(X^R) |X^R|^{p-2}_{L^p_x} \times \big| \dfrac{\partial X^R}{\partial x} \big|_{L^p_x} \bigg|_{L^r_\omega L^{r'}_{[0,T]}}\\
&= \bigg| \bigg[  \int_0^T \bigg( \theta_R(X^R(s)) |X^R(s)|^{p-2}_{L^p_x} \bigg)^{\frac{r}{r-1}} \bigg| \dfrac{\partial X^R}{\partial x}(s) \bigg|^{\frac{r}{r-1}}_{L^p_x} \, ds \bigg]^{1/r'} \bigg|_{L^r_\omega}
\end{aligned}
\end{equation*}

\noindent Applying now H\"{o}lder inequality for the integral w.r.t. $t$, 
\begin{align*}
&\leq \bigg| \bigg[ \,  \Big| \Big( \theta_R(X^R) |X^R|^{p-2}_{L^p_x} \Big)^{\frac{r}{r-1}} \Big|_{L^{\frac{r-1}{r-2}}[0,T]} \times \bigg| \Big| \dfrac{\partial X^R}{\partial x} \Big|^{\frac{r}{r-1}}_{L^p_x} \bigg|_{L^{r-1}[0,T]} \bigg]^{\frac{r-1}{r}} \bigg|_{L^r_\omega}\\
& = \bigg| \, \bigg[ \int_0^T \Big( \theta_R (X^R(s)) \Big)^{\frac{r}{r-2}} |X^R(s)|^{\frac{(p-2)r}{r-2}}_{L^p_x} \, ds \bigg]^{\frac{r-2}{r}}  \times \Big| \dfrac{\partial X^R}{\partial x} \Big|_{L^r([0,T];L^p_x)} \bigg|_{L^r_\omega}.
\end{align*}

\noindent Since $p-2 \leq r-2$, and setting $\widetilde{t}(\omega):= \inf \{s: \, |X^R(\omega)|_{L^r([0,s];L^p_x)} \geq 2R \}$, we get
\begin{align} \label{thetaestimate}
\nonumber & \leq \bigg| \, \bigg[ \int_0^T \Big( \theta_R (X^R(s)) \Big)^{\frac{r}{r-2}} |X^R(s)|^r_{L^p_x} \, ds \bigg]^{\frac{r-2}{r}}  \times \Big| \dfrac{\partial X^R}{\partial x} \Big|_{L^r([0,T];L^p_x)} \bigg|_{L^r_\omega} \\ \nonumber
& \leq \bigg| \, \bigg[ \int_0^{\widetilde{t}} |X^R(s)|^r_{L^p_x} \, ds \bigg]^{\frac{r-2}{r}}   \times \Big| \dfrac{\partial X^R}{\partial x} \Big|_{L^r([0,T];L^p_x)} \bigg|_{L^r_\omega} \\
& = \bigg| \, |X^R|^{r-2}_{L^r([0,\widetilde{t}];L^p_x)} \times \Big| \dfrac{\partial X^R}{\partial x} \Big|_{L^r([0,T];L^p_x)} \bigg|_{L^r_\omega} = (2R)^{r-2} \Big| \dfrac{\partial X^R}{\partial x} \Big|_{L^r_{\omega}(L^r([0,T];L^p_x))}.
\end{align}

\noindent Hence, 
\begin{align} \label{estimateinformnorm}
|\mathcal{T}(X^R)|_{L^r_\omega(C([0,T];D(\mathcal{E}_{\star})))} \nonumber & \leq |X_0|_{L^r_\omega(D(\mathcal{E}_{\star})))} + C\Big(c_1,r,p,\sum_{l=1}^{n_+(AB^{\dag})} \| \varphi_l \|_{L^p_x} \| \varphi_l \|_{L^2_x} \Big) T^{\hat{\beta}} (2R)^{2 \sigma} |X^R|_{L^r_\omega L^{r}_{[0,T]} L^{p}_x} \\
& +  3(2 \sigma +1) C(c_1,r,p) T^{\overline{\beta}} (2R)^{r-2}  \Big| \dfrac{\partial X^R}{\partial x} \Big|_{L^r_{\omega}(L^r([0,T];L^p_x))}.
\end{align}

\noindent We proceed now with the ${L^r_{\omega}(L^r([0,T];W^{1,p}_x))}$--norm:
\begin{equation*}
\begin{aligned}
|\mathcal{T}(X^R) & |_{ L^r_{\omega} L^r_{[0,T]}W^{1,p}_x} \\ 
& \leq |S_{\beta}(\cdot)X_0|_{L^r_{\omega} L^r_{[0,T]}W^{1,p}_x} + \Big|\int_0^t S_{\beta}(t-s) \theta_R(X^R(s)) |X^R|^{2 \sigma}X^R(s) \, ds \Big|_{L^r_{\omega} L^r_{[0,T]}W^{1,p}_x} \\
& = |S_{\beta}(\cdot)X_0|_{L^r_{\omega}(L^r([0,T];L^{p}_x))} + \Big|\dfrac{\partial}{\partial x} S_{\beta}(\cdot)X_0 \Big|_{L^r_{\omega}L^r_{[0,T]}L^{p}_x} \\
& \quad + \Big|\int_0^t S_{\beta}(t,s) \theta_R(X^R(s)) |X^R|^{2 \sigma}X^R(s) \, ds \Big|_{L^r_{\omega}L^r_{[0,T]}L^{p}_x} \\
& \quad  +  \Big| \dfrac{\partial}{\partial x} \int_0^t S_{\beta}(t,s) \theta_R(X^R(s)) |X^R|^{2 \sigma}X^R(s) \, ds \Big|_{L^r_{\omega}L^r_{[0,T]}L^{p}_x} =: I_1 + \widetilde{I_1} + I_2 + \widetilde{I_2}
\end{aligned}
\end{equation*}

\noindent For the first and the third term, by \eqref{LrLrLpalready} we already have that
\begin{equation*}
I_1+I_2 \leq c_{r,p}\; T^{\beta/2}\; |X_0|_{L^r_\omega L_x^2}+c_{r,p} \;T^{\tilde{\beta}}\; (2R)^{2\sigma} \; \big| X^R \big| _{L^r_{\omega}L^r_{[0,T]}L^{p}_x}.
\end{equation*}

\noindent In view of Proposition \ref{derivativesemigroup}, Theorem \ref{stochastic strichartz} (i) and \cite[Prop 3.10]{debouarddebussche}, and the estimate \eqref{Pp Strichartz}
\begin{equation*}
\begin{aligned}
\widetilde{I_1}
& \leq C \Big( c_1, r, p, \sum_{l=1}^{n_+(AB^{\dag})} \| \varphi_l \|_{L^2_x} \| \varphi'_l \|_{L^p_x}  \Big) T^{\beta/2} |X_0|_{L^r_\omega(D(\mathcal{E}_{\star}))}.
\end{aligned}
\end{equation*}

\noindent In the case of the last term, by Proposition \ref{derivativesemigroup}, Theorem \ref{stochastic strichartz} (ii) and \cite[Prop 3.10]{debouarddebussche}, and the estimates \eqref{LrLrLpalready} and \eqref{thetaestimate}, we get
\begin{equation*}
\begin{aligned}
\widetilde{I_2}
& \leq 3 (2 \sigma + 1) c_{r,p} T^{\overline{\beta}}  (2R)^{r-2} \Big| \dfrac{\partial X^R}{\partial x} \Big|_{L^r_{\omega}L^r_{[0,T]}L^{p}_x} + \sum_{l=1}^{n_+(AB^{\dag})} \| \varphi_l \|_{L^p_x} \| \varphi'_l \|_{L^p_x} T^{2/r} (2R)^{2 \sigma} |X^R|_{L^r_{\omega}L^r_{[0,T]}L^{p}_x}
\end{aligned}
\end{equation*}

\noindent Thus, we have that 
\begin{equation*}
\begin{aligned}
|\mathcal{T}&(X^R)  |_{ L^r_{\omega}L^r_{[0,T]}W^{1,p}_x} \leq  C \Big( c_1, r, p, \sum_{l=1}^{n_+(AB^{\dag})} \| \varphi_l \|_{L^2_x} \| \varphi'_l \|_{L^p_x}  \Big) T^{\beta/2} |X_0|_{L^r_\omega(D(\mathcal{E}_{\star}))}\\
& + c_{r,p} \max \Big[ T^{\tilde{\beta}} (2R)^{2\sigma} , 3 (2 \sigma + 1) T^{\overline{\beta}}  (2R)^{r-2},  \sum_{l=1}^{n_+(AB^{\dag})} \| \varphi_l \|_{L^p_x} \| \varphi'_l \|_{L^p_x} T^{2/r} (2R)^{2 \sigma} \Big] |X^R|_{L^r_{\omega}L^r_{[0,T]}W^{1,p}_x}
\end{aligned}
\end{equation*}

\noindent Corroborating this with \eqref{estimateinformnorm}, we finally obtain
\begin{equation*}
\begin{aligned}
& \vertiii{\mathcal{T}(X^R)}  \leq \bigg[ 1+ C \Big( c_1, r, p, \sum_{l=1}^{n_+(AB^{\dag})} \| \varphi_l \|_{L^2_x} \| \varphi'_l \|_{L^p_x}  \Big) T^{\beta/2} \bigg] |X_0|_{L^r_\omega(D(\mathcal{E}_{\star})))} + c_{r,p} \max \Big[ T^{\tilde{\beta}} (2R)^{2\sigma} , \\
&  3 (2 \sigma + 1) T^{\overline{\beta}}  (2R)^{r-2},  \sum_{l=1}^{n_+(AB^{\dag})} \| \varphi_l \|_{L^p_x} \| \varphi'_l \|_{L^p_x} T^{2/r} (2R)^{2 \sigma}, \sum_{l=1}^{n_+(AB^{\dag})} \| \varphi_l \|_{L^p_x} \| \varphi_l \|_{L^2_x} T^{\hat{\beta}} (2R)^{2 \sigma} \Big] |X^R|_{L^r_{\omega}L^r_{[0,T]}W^{1,p}_x}
\end{aligned}
\end{equation*}
Choosing now $a$ to be twice of the first term of the r.h.s. of the previous inequality, there exists $T \leq  1$ not depending on $|X_0|_{L^2_{\omega};D(\mathcal{E}_{\star})}$ s.t. the coefficient of $ |X^R|_{L^r_{\omega}(L^r([0,T];W^{1,p}_x))}$ is less than $1/2$, and hence $\mathcal{T}$ is well-defined from $\widetilde{E}(T,a)$ to $\widetilde{E}(T,a)$. Since closed balls of $L^r_\omega(C([0,T];D(\mathcal{E}_{\star}))) \cap L^r_{\mathcal{P}}(\Omega \times [0,T];W^{1,p}_x)$ are closed in $L^r_{\mathcal{P}}(\Omega \times [0,T];L^{p}_x)$, the fixed point of $\mathcal{T}$ obtained in Theorem \ref{solution truncated} belongs in fact to $L^r_\omega(C([0,T];D(\mathcal{E}_{\star}))) \cap L^r_{\mathcal{P}}(\Omega \times [0,T];W^{1,p}_x)$. By iteration, we get that $X^R$ has $\mbox{a.s.}$ paths in $C(\mathbb{R_+;D(\mathcal{E}_{\star})})$.

To get the desired result for $X$ knowing it holds for $X^R$, we use localization: let $T_R:=\inf \{ t \geq 0: \ |X|_{L^r([0,t];L^p_x)} \geq R \}$, where $X$ is the $L^2-$solution to \eqref{mild} given by Theorem \ref{wellposednessL2}, in particular $T_R \underset{k \to \infty}\longrightarrow \infty \mbox{a.s.}$. Now the result follows since on $[0,T_R]$ we have $X=X^R$. 
\end{proof}

\begin{proof}[\bf Proof of Theorem \ref{convergence3}]
First of all, for each $\varepsilon, R >0$, let $X_\varepsilon^{R}$ be the solution to \eqref{eq 2.13}, given by Theorem \ref{well-posedness}. 
By Theorem \ref{convergence}, for each $R$, $(X_\varepsilon^{R})_{\varepsilon>0}$ converges in distribution to the solution $X^{R}$ of \eqref{eq 2.14} when $\varepsilon \to 0$, in $C([0,T]; D(\mathcal{E}_\star))$, and by Skorohod Theorem, after a change of probability, there is no loss if we assume that the convergence holds $\mathbb{P}$-a.s.
Set
\begin{equation*}
\begin{aligned}
&\tau_\varepsilon^{R}:=\inf \{t>0 : |X_\varepsilon^{R}|_{L^{\infty}(\Gamma_{\star})}\geq \sqrt{R} \}\\
&\tau^{R}:=\inf \{t>0 : |X^{R}|_{L^{\infty}(\Gamma_{\star})}\geq \sqrt{R} \}\\
&\tau_\varepsilon=\lim\limits_{R\to \infty} \tau_\varepsilon^{R}.
\end{aligned}
\end{equation*}
Now, for each $\varepsilon>0$, $(X_\varepsilon^{R})_{R>0}$ can be superposed to obtain a unique local solution $X_\varepsilon$ to \eqref{approximate} in $C([0,T];D(\mathcal{E}_\star))$, on $[0, \tau_\varepsilon)$.
Also, by Theorem \ref{thm 2.16}, equation \eqref{mild} has a unique solution $X$ with paths a.s. in $C([0,T]; D(\mathcal{E}_\ast))$, hence it must coincide with $X^{R}$ on $[0, \tau^{R})$, a.s.
In particular
\begin{equation} \label{eq 4.19}
\lim\limits_{R\to \infty} \mathbb{P}([\tau^{R}<T])=0.
\end{equation}
Further, since $|\cdot|_{D(\mathcal{E}_\ast)}$ is equivalent with the $H^{1}$-norm on $\Gamma_\star$, we have the continuous embedding $D(\mathcal{E}_\ast)\subset L^{\infty}(\Gamma_\star)$.
Hence $(X_{\varepsilon}^{R})_{\varepsilon>0}$ converges a.s. to $X^{R}$ in $C([0,T]; L^{\infty}(\Gamma_\star))$, and consequently, for all $R\geq 0$ there exists $\varepsilon_R$ s.t.
\begin{equation*} \label{eq 4.20}
\tau_\varepsilon\geq \tau_\varepsilon^{R+1}\geq \tau^{R} \mbox{ for all } \varepsilon \geq \varepsilon_R.   
\end{equation*}
By \eqref{eq 4.19} we get
$\lim\limits_{\varepsilon\to 0} \mathbb{P}([\tau_\varepsilon<T])=0,$
and finally,
\begin{align*}
\lim\limits_{\varepsilon\to 0}\mathbb{P}(&|X_\varepsilon1_{[0,\tau_\varepsilon)}-X|_{C([0,T];D(\mathcal{E}_\star))}\geq \delta) \\
&\leq \lim\limits_{\varepsilon\to 0}\mathbb{P}(\tau^{R}\geq T,\; |X_\varepsilon1_{[0,\tau_\varepsilon)}-X|_{C([0,T];D(\mathcal{E}_\star))}\geq \delta) +  \mathbb{P}(\tau^{R}<T)\\
&= \lim\limits_{\varepsilon\to 0}\mathbb{P}(\tau^{R}\geq T,\; |X^{R}_\varepsilon1_{[0,\tau_\varepsilon)}-X^{R}|_{C([0,T];D(\mathcal{E}_\star))}\geq \delta) +  \mathbb{P}(\tau^{R}<T)\\
& = \mathbb{P}(\tau^{R}<T) \mathop{\longrightarrow}\limits_{R\to \infty} 0. \qedhere
\end{align*}
\end{proof}

\begin{proof}[\bf Proof of Proposition \ref{derivativesemigroupdelta}] In the case of $\delta-$type condition \eqref{delta}, there is non-zero Robin part. More precisely, by \cite[Subsection 1.4.4]{berkolaiko}, $P_R=\mathbb{I}_n - P_D$ which is the orthogonal projection onto the space of vectors with equal coordinates, and $\Lambda$ acts as the multiplication with $\alpha/n$. Then, by \eqref{otherGexpression},
\begin{equation*}
\{ [ \mathbb{I}_n -G(k;A,B)] u_{0}(0)\}_i = \dfrac{2 \alpha}{\alpha - \mathrm{i}n k} u_{0,1}(0).
\end{equation*}

\noindent Hence, by \cite[Identity (2.15)]{adami}, the last term in \eqref{derivativesemigroupwhole} equals
\begin{align}
\nonumber & \lim_{\varepsilon \to 0}  \dfrac{1}{2 \pi} \int_{\mathbb{R}} \mathrm{e}^{-(\mathrm{i} t + \varepsilon) k^2} \mathrm{e}^{\mathrm{i}k x_i} \{ [ \mathbb{I}_n -G(k;A,B)] u_{0}(0)\}_i \, dk = 2 \alpha u_{0,1}(0) \lim_{\varepsilon \to 0}  \dfrac{1 }{2 \pi} \int_{\mathbb{R}} \dfrac{\mathrm{e}^{-(\mathrm{i} t + \varepsilon) k^2} \mathrm{e}^{\mathrm{i}k x_i}}{\alpha -  \mathrm{i}  n k} \, dk\\ \nonumber
& = \dfrac{2 \alpha}{n} u_{0,1}(0) \int_0^{\infty} \mathrm{e}^{-\frac{\alpha}{n} u} \dfrac{e^{\mathrm{i \frac{(x_i + u)^2}{4 t}}}}{\sqrt{4 \pi \mathrm{i} t}} \, du = \dfrac{2 \alpha}{n} u_{0,1}(0) (\mathrm{e}^{- \mathrm{i} t \Delta_{\mathbb{R}}} \psi)(x_i). \qedhere
\end{align}
\end{proof}

\section*{Acknowledgement}
The second named author was partially supported by PhD fellowships of University of Bucharest and L’Agence Universitaire de la Francophonie and by  a BITDEFENDER Junior Research Fellowship from "Simion Stoilow" Institute of Mathematics of the Romanian Academy.

We would like to thank Prof. Liviu Ignat for introducing the authors to this subject and for helpful discussions.

\bibliographystyle{abbrv}

\addcontentsline{toc}{section}{\refname}

{\footnotesize\bibliography{bibfileia}}

\end{document}